\documentclass[11pt,final]{article}
\usepackage{a4}
\usepackage{amsmath}%
\usepackage{amstext}%
\usepackage{amssymb}%
\usepackage{comment}
\usepackage{listings}

\usepackage[final]{graphicx}
\usepackage{subcaption}

\usepackage[margin=0.5in]{geometry}

\newtheorem{theorem}{Theorem}

\newtheorem{axiom}{Axiom}

\newtheorem{corollary}[theorem]{Corollary}

\newtheorem{definition}[axiom]{Definition}

\newtheorem{lemma}[theorem]{Lemma}

\newtheorem{proposition}[theorem]{Proposition}
\newtheorem{rem}[theorem]{Remark}
\newenvironment{remark}{\rem\rm}{\endrem}

\newcounter{unnumber}

\newtheorem{prf}[unnumber]{Proof.}
\newenvironment{proof}{\prf\rm}{\hfill{$\blacksquare$}\endprf}
\newcommand{\R}{\mathbb{R}}%
\newcommand{\N}{\mathbb{N}}%
\newcommand{\e}{\varepsilon}%
\newcommand{\ol}{\overline}%

\newcommand{\n}{{\nabla}}

\newcommand{\ds}{\displaystyle}
\newcommand{\To}{\longrightarrow}
\def\a{\alpha}

\def\b{\beta}
\def\e{\epsilon}
\def\d{\delta}
\def\t{\theta}
\def\g{\gamma}

\def\<{\langle}
\def\>{\rangle}
\DeclareMathOperator*\prox{prox}%
\DeclareMathOperator*\argmin{argmin}

\DeclareMathOperator*\crit{crit}
\DeclareMathOperator*\dist{dist}

\author{ Szil\'{a}rd Csaba L\'{a}szl\'{o} \thanks{Technical University of Cluj-Napoca, Department of Mathematics,
 Str. Memorandumului nr. 28, 400114 Cluj-Napoca, Romania, e-mail: laszlosziszi@yahoo.com This work was supported by a grant of Ministry of Research and Innovation, CNCS - UEFISCDI, project number PN-III-P1-1.1-TE-2016-0266, and  by a grant of Ministry of Research and Innovation,
 CNCS - UEFISCDI, project number PN-III-P4-ID-PCE-2016-0190, within PNCDI III.}}
\title{Convergence rates for an inertial algorithm of gradient type associated to a smooth non-convex minimization}

\begin{document}
\maketitle

\noindent \textbf{Abstract.}  We investigate an inertial algorithm of gradient type  in connection with the minimization of a non-convex differentiable function. The algorithm is formulated in the spirit of Nesterov's accelerated convex gradient method. We prove some abstract convergence results which applied to our numerical scheme allow us to show that the generated sequences converge to a critical point of the objective function, provided a regularization of the objective function satisfies the Kurdyka-{\L}ojasiewicz property. Further, we obtain convergence rates for the generated sequences and the objective function values formulated in terms of the {\L}ojasiewicz exponent of a regularization of the objective function. Finally, some numerical experiments are presented in order to compare our numerical scheme and some algorithms well known  in the literature. \vspace{1ex}

 \vspace{1ex}


\section{Introduction}
Inertial optimization algorithms deserve special attention in both convex and non-convex optimization due to their better convergence rates compared to non-inertial ones, as well as due to their ability to detect multiple critical points of non-convex functions via an appropriate control of the inertial parameter \cite{alvarez-attouch2001,AGR,bete09,BCH,ch-do2015,CG,GFJ,LRP,PL,nesterov83,poly,SYLHGJ}.  Non-inertial methods lack the latter property \cite{BCL}.

With the growing use of non-convex objective functions in some applied fields, such as image processing or machine learning, the need for non-convex numerical methods increased significantly. However, the literature of non-convex optimization methods is still very poor, we refer to \cite{OCBP} (see also \cite{Ochs}), \cite{BCL1} and \cite{ZK}  for some algorithms that can be seen as extensions of Polyak's heavy ball method \cite{poly} to the non-convex case and the papers \cite{att-b-red-soub2010} and \cite{att-b-sv2013}  for some abstract non-convex methods.

In this paper we investigate  an algorithm, with a possible non-convex objective function, which has a  form similar to Nesterov's accelerated convex gradient method  \cite{nesterov83,ch-do2015}.   %However, this scope is achieved only partially, our algorithm is actually a mixture of Polyak's heavy ball method and Nesterov's accelerated  gradient method, extended to a non-convex setting.

Let $g:\R^m\To \R$ be a (not necessarily convex) Fr\'{e}chet differentiable function with $L_g$-Lipschitz continuous gradient, that is, there exists $L_g\ge 0$ such that $\|\n g(x)-\n g(y)\|\le L_g\|x-y\|$ for all $x,y\in \R^m.$ We deal with the optimization problem
\begin{equation}\label{opt-pb} (P) \ \inf_{x\in\R^m}g(x). \end{equation}

Of course regarding this possible non-convex optimization problem, in contrast to the convex case where every local minimum is also a global one, we are interested to approximate the critical points of the objective function $g$. To this end we associate to the optimization problem \eqref{opt-pb} the following inertial algorithm of gradient type.
Consider the starting points $x_0,x_{-1}\in\R^m$ and for all $n\in\N$ let
\begin{equation}\label{generaldiscrete}\left\{\begin{array}{lll}
\ds y_n=x_n+\frac{\b n}{n+\a}(x_n-x_{n-1})\\
\\
\ds x_{n+1}=y_n-s\n g(y_n),
\end{array}\right.
\end{equation}
where $\a>0,\,\b\in(0,1)$ and $0<s<\frac{2(1-\b)}{L_g}.$

We underline that the main difference between Algorithm \eqref{generaldiscrete} and the already mentioned non-convex versions of the heavy ball method studied in \cite{OCBP} and \cite{BCL1} is the same as the difference between the methods of Polyak  \cite{poly} and Nesterov \cite{nesterov83}, that is, meanwhile the first one evaluates the gradient in $x_n$ the second one evaluates the gradient in $y_n.$  One can  observe at once the similarity between the formulation of Algorithm \eqref{generaldiscrete}  and the algorithm considered by Chambolle and Dossal in \cite{ch-do2015} (see also \cite{AAD,AD,att-c-p-r-math-pr2018}) in order to prove the convergence of the iterates of the modified FISTA algorithm \cite{bete09}. Indeed, the algorithm studied by Chambolle and Dossal in the context of a convex optimization problem can be obtained from Algorithm \eqref{generaldiscrete} by violating its assumptions and allowing  the case $\b=1$ and $s\le \frac{1}{L_g}.$

Unfortunately, due to the form of the stepsize $s$,  we cannot allow the case  $\b=1$ in Algorithm \eqref{generaldiscrete}, but what is lost at the inertial parameter it is gained at the stepsize, since  in the case $\b<\frac12$ one may allow a better stepsize than $\frac{1}{L_g}$, more precisely the stepsize in Algorithm \eqref{generaldiscrete} satisfies $s\in\left(\frac{1}{L_g},\frac{2}{L_g}\right)$.

Let us mention that to our knowledge Algorithm \eqref{generaldiscrete} is  the first attempt  in the literature to extend the Nesterov accelerated convex gradient method to the case when the objective function $g$ is possible non-convex.

Another interesting fact about Algorithm \eqref{generaldiscrete} which enlightens the relation with Nesterov's accelerated convex gradient  method  is that both methods are modeled by the same differential equation that governs the so called continuous heavy ball system with vanishing damping, that is,
\begin{equation}\label{ee11}
\ddot{x}(t)+\frac{\a}{t}\dot{x}(t)+\n g(x(t))=0.
\end{equation}

We recall that \eqref{ee11} (with $\a=3$) has been introduced by  Su, Boyd and Cand\`es  in \cite{su-boyd-candes} as the continuous counterpart of Nesterov's accelerated gradient  method and from then it was  the  subject  of  an intensive  research. Attouch and his co-authors \cite{att-c-p-r-math-pr2018,att-p-r-jde2016} proved  that  if $\a>3$ in \eqref{ee11}  then  a  generated  trajectory $x(t)$  converges  to  a  minimizer  of the convex objective function $g$ as $t\To+\infty$, meanwhile the convergence rate of the objective function along the trajectory is $o(1/t^2)$.  Further, in \cite{att-c-r-arx2017} some results concerning the convergence rate of the convex objective function  $g$ along  the  trajectory generated by \eqref{ee11} in  the  subcritical  case $\a\le 3$ have been obtained.

 In one hand, in order to obtain optimal convergence rates of the trajectories generated by \eqref{ee11},  Aujol, Dossal and Rondepierre \cite{ADR} assumed that beside convexity the objective function $g$ satisfies also some geometrical conditions, such as the {\L}ojasiewicz property. %The importance of their results obtained in \cite{ADR} is emphasized by the fact that applying the classical Nesterov scheme on a convex objective function without studying its geometrical properties may lead to sub-optimal algorithms.

On the other hand, Aujol and Dossal  obtained in \cite{AD} some general convergence rates and also the convergence of the trajectories generated by \eqref{ee11} to a minimizer of the objective function $g$ by dropping the convexity assumption on $g$ but assuming that the function $(g(x(t))-g(x^*))^\b$ is convex, where $\b$ is strongly related to the damping parameter $\a$ and $x^*$ is a global minimizer of $g.$ In case $\b=1$ they results reduce to the results obtained in \cite{att-c-p-r-math-pr2018,att-p-r-jde2016,att-c-r-arx2017}.

  However, the convergence of the  trajectories generated by  the continuous heavy ball system with vanishing damping  in the general case when the objective function $g$ is possible non-convex is still an open question.
 Some important steps in this direction have been made in \cite{BCL-AA} (see also \cite{BCL}), where convergence of the trajectories of a system, that can be viewed as a perturbation of \eqref{ee11}, have been obtained in a non-convex setting. More precisely, in \cite{BCL-AA}  is considered the continuous dynamical system
  \begin{equation}\label{eee11}
\ddot{x}(t)+\left(\g+\frac{\a}{t}\right)\dot{x}(t)+\n g(x(t))=0,\, x(t_0)=u_0,\,\dot{x}(t_0)=v_0,
\end{equation}
where $t_0>0,\,u_0,\,v_0\in\R^m,\,\g>0$ and $\a\in\R.$  Note that here $\a$ can take nonpositive values. For $\a=0$ we recover the dynamical system studied in \cite{BBJ}. According to \cite{BCL-AA} the trajectory  generated by the dynamical system \eqref{eee11} converges  to a critical point of $g$ if a regularization of $g$ satisfies the Kurdyka-{\L}ojasiewicz property.

The connection between the continuous dynamical system \eqref{eee11} and Algorithm \eqref{generaldiscrete} is that the latter one can be obtained via discretization from  \eqref{eee11}, as it is shown in Appendix. Further, following the same approach as Su, Boyd and Cand\`es in \cite{su-boyd-candes} (see also \cite{BCL-AA}), we show in Appendix that  by choosing  appropriate values of $\b$ the numerical scheme \eqref{generaldiscrete}  has the exact limit  the continuous second order dynamical system governed by  \eqref{ee11}  and also the continuous dynamical system \eqref{eee11}.
Consequently, our numerical scheme \eqref{generaldiscrete} can be seen as the discrete counterpart of the continuous dynamical systems \eqref{ee11} and \eqref{eee11} in a full non-convex setting.
%studied in \cite{su-boyd-candes, att-c-p-r-math-pr2018,att-p-r-jde2016,att-c-r-arx2017},

The paper is organized as follows. In the next section we prove an abstract convergence result that may become useful in the future in the context of related researches. Our result is formulated in the spirit of the abstract convergence result from \cite{att-b-sv2013}, however it can also be used in the case when we evaluate de gradient of the objective function in iterations that contain inertial terms. Further, we apply the abstract convergence result obtained to \eqref{generaldiscrete} by showing that its assumptions are satisfied by the sequences generated by the numerical scheme \eqref{generaldiscrete}, see also \cite{att-b-sv2013,b-sab-teb,BCL1}.
In  section 3 we obtain several convergence rates both for the sequences $(x_n)_{n\in\N}$ and $(y_n)_{n\in\N}$ generated by the numerical scheme \eqref{generaldiscrete}, as well as for the function values $g(x_n)$ and $g(y_n)$ in the terms of the {\L}ojasiewicz exponent of the objective function $g$ and a regularization of $g$, respectively (for some general results see  \cite{FGP,GRV}). As an immediate consequence we obtain linear convergence rates in the case when the objective function is strongly convex.  Further, in section 4 via some numerical experiments we show that Algorithm \eqref{generaldiscrete} has a very good behavior compared with some well known algorithms from the literature.
Finally, in Appendix we show that Algorithm \eqref{generaldiscrete} and the second order differential equations  \eqref{ee11} and \eqref{eee11} are strongly connected.

\section{Convergence analysis}

The central question that we are concerned in this section regards the convergence of the sequences generated by the numerical method \eqref{generaldiscrete} to a critical point of the objective function $g,$ which in the non-convex case critically depends on the Kurdyka-{\L}ojasiewicz property \cite{lojasiewicz1963,kurdyka1998} of an appropriate regularization of the objective function. The Kurdyka-{\L}ojasiewicz property is a key tool in non-convex optimization (see  \cite{attouch-bolte2009,att-b-red-soub2010,att-b-sv2013,b-sab-teb,BNPS,BC,BC1,BC2,BCL,BCL1,BCL-AA,CPR,FGP,HJ,OCBP,simon}), and might look restrictive, but from a practical point of view in problems appearing in image processing, computer vision or machine learning this property is always satisfied.

 We prove at first an abstract convergence result which applied to Algorithm \eqref{generaldiscrete} ensures the  convergence of the  generated sequences. % to a critical point of  the objective function $g$.
The main result of this section is the following.
\begin{theorem}\label{convergence} In the settings of problem \eqref{opt-pb}, for some starting points $x_0,x_{-1}\in\R^m,$ consider the sequence $(x_n)_{n\in\N}$ generated by Algorithm \eqref{generaldiscrete}. Assume that $g$ is bounded from below and consider the function
$$H:\R^m\times\R^m\To\R,\,H(x,y)=g(x)+\frac12\|y-x\|^2.$$ Let $x^*$ be a cluster point of the sequence $(x_n)_{n\in\N}$ and assume that $H$ has the Kurdyka-{\L}ojasiewicz property at a $z^*=(x^*,x^*).$
%Then, the sequence $(x_n)_{n\in\N}$ converges to $x^*$, $(z_n)_{n\in\N}$ converges to $z^*$ and $z^*\in\crit(F).$

 Then, the sequence $(x_n)_{n\in\N}$ converges to $x^*$ and $x^*$ is a critical point of the objective function $g.$
\end{theorem}

%In this section we investigate the convergence of the proposed algorithm. We show that the sequences generated by the numerical scheme \eqref{generaldiscrete} converge  to a critical point of the objective function $g$, provided the regularization of $g$, $H(x,y)=g(x)+\frac12\|y-x\|^2,$ satisfies the Kurdyka-{\L}ojasiewicz property at a cluster point of $(x_n)_{n\in\N}$.

\subsection{An abstract convergence result}

In what follows, by using some similar techniques as in \cite{att-b-sv2013}, we  prove an abstract convergence result. For other works where these techniques were used we refer to \cite{FGP,OCBP}. %Our result might become useful in the future for obtaining the convergence of an inertial algorithm in the non-convex setting.

Let us denote by $\omega((x_n)_{n\in\N})$ the set of cluster points of the sequence $(x_n)_{n\in\N}\subseteq\R^m,$ that is,
$$\omega((x_n)_{n\in\N}):=\left\{x^*\in\R^m:\mbox{ there exists a subsequence }(x_{n_j})_{j\in\N}\subseteq(x_n)_{n\in\N}\mbox{ such that }\lim_{j\To+\infty}x_{n_j}=x^*\right\}.$$

Further, we denote by  $\crit(g)$ the set of critical points of a smooth function $g:\R^m\To\R$, that is,
$$\crit(g):=\{x\in\R^m: \nabla g(x)=0\}.$$

In order to continue our analysis we need the concept of a KL function. For $\eta\in(0,+\infty]$, we denote by $\Theta_{\eta}$
the class of concave and continuous functions $\varphi:[0,\eta)\To [0,+\infty)$ such that $\varphi(0)=0$, $\varphi$ is
continuously differentiable on $(0,\eta)$, continuous at $0$ and $\varphi'(s)>0$ for all
$s\in(0, \eta)$.

\begin{definition}\label{KL-property} \rm({\it Kurdyka-{\L}ojasiewicz property}) Let $g:\R^m\To \R$ be a differentiable function. We say that $g$ satisfies the {\it Kurdyka-{\L}ojasiewicz (KL) property} at
$\ol x\in \R^m$
if there exist $\eta \in(0,+\infty]$, a neighborhood $U$ of $\ol x$ and a function $\varphi\in \Theta_{\eta}$ such that for all $x$ in the
intersection
$$U\cap \{x\in\R^m: g(\ol x)<g(x)<g(\ol x)+\eta\}$$ the following, so called KL inequality, holds
\begin{equation}\label{KLineq}
\varphi'(g(x)-g(\ol x))\|\n g(x)\|\geq 1.
\end{equation}
If $g$ satisfies the KL property at each point in $\R^m$, then $g$ is called a {\it KL function}.
\end{definition}

Of course, if $g(\ol x)=0$ then the previous inequality can be written as
$$\|\n(\varphi\circ g)(x)\|\ge 1.$$

The origins of this notion go back to the pioneering work of {\L}ojasiewicz \cite{lojasiewicz1963}, where it is proved that for a
real-analytic function $g:\R^m\To\R$ and a critical point $\ol x\in\R^m$  there exists $\theta\in[1/2,1)$ such that the function
$x\rightarrowtail |g(x)-g(\ol x)|^{\theta}\|\nabla g(x)\|^{-1}$ is bounded around $\ol x$. This corresponds to the situation when
$\varphi(s)=C(1-\theta)^{-1}s^{1-\theta}$. The result of {\L}ojasiewicz allows the interpretation of the KL property as a re-parametrization of the function values in order to avoid flatness around the critical points, therefore $\varphi$ is called a desingularizing function \cite{BBJ}. Kurdyka \cite{kurdyka1998} extended this property to differentiable functions definable in an o-minimal structure.
Further extensions to the nonsmooth setting can be found in \cite{ att-b-red-soub2010, b-d-l2006, b-d-l-s2007, b-d-l-m2010,DM}.

To the class of KL functions belong semi-algebraic, real sub-analytic, semi-convex, uniformly convex and
convex functions satisfying a growth condition. We refer the reader to
\cite{ attouch-bolte2009,att-b-red-soub2010,att-b-sv2013,b-sab-teb, b-d-l2006,b-d-l-s2007,b-d-l-m2010} and the references therein  for more details regarding all the classes mentioned above and illustrating examples.

In what follows we formulate some conditions that beside the KL property at a point of a continuously differentiable function lead to a convergence result.
Consider a sequence $(x_n)_{n\in\N}\subseteq\R^m$ and fix the positive constants $a,b> 0,\,c_1,c_2\ge0,\,c_1^2+c_2^2\neq 0.$ Let $F : \R^m\times\R^m\To\R$ be a continuously Fr\'echet differentiable function. Consider further  a sequence $(z_n)_{n\in\N} := (v_n,w_n)_{n\in\N}\subseteq\R^m\times\R^m$ which is related to the sequence $(x_n)_{n\in\N}$    via the conditions (H1)-(H3) below.
  %Then, the conditions we require for $(z_n)_{n\in\N}$ are:
\vskip0.3cm
(H1) For each $n\in\N$ it holds
$$ a\|x_{n+1}-x_{n}\|^2\le F(z_{n})-F(z_{n+1}).$$

(H2) For each $n \in\N,\,n\ge 1$ one has
$$\|\n F(z_n)\|\le b(\|x_{n+1}-x_{n}\|+\|x_n-x_{n-1}\|).$$

(H3)   For each $n \in\N\,,\,n\ge 1$ and every $z=(x,x)\in\R^m\times\R^m$ one has
 $$\|z_n-z\|\le c_1\|x_n-x\|+c_2\|x_{n-1}-x\|.$$

\begin{remark} One can observe that the conditions (H1) and (H2) are very similar to those in \cite{att-b-sv2013}, \cite{FGP} and \cite{OCBP}, however there are some major differences. First of all observe that the conditions in \cite{att-b-sv2013} or \cite{FGP} can be rewritten into our setting by considering that the sequence $(z_n)_{n\in\N}$ has the form $z_n=(x_n,x_{n})$ for all $n\in\N$ and the lower semicontinuous function $f$ considered in \cite{att-b-sv2013} satisfies $f(x_n)=F(z_n)$ for all $n\in\N.$

Further, in \cite{OCBP} the sequence $(z_n)_{n\in\N}$ has the special form  $z_n=(x_n,x_{n-1})$  for all $n\in\N.$

\begin{itemize}
\item Our condition (H3) is automatically satisfied for the sequence considered in \cite{att-b-sv2013} that is $z_n=(x_n,x_n)$ with $c_1=\sqrt{2},\,c_2=0$ and also for the sequence considered in \cite{OCBP} $z_n=(x_n,x_{n-1})$ with $c_1=c_2=1.$
\item In \cite{att-b-sv2013} and \cite{FGP} the condition (H1) reads as
$$ a_n\|x_{n+1}-x_{n}\|^2\le F(z_n)-F(z_{n+1}),$$
 where $a_n=a>0$ in \cite{att-b-sv2013} and $a_n>0$ in \cite{FGP}, which are formally  identical to our assumption but our sequence $z_n$ has a more general form, meanwhile in \cite{OCBP} (H1) is
$$ a\|x_n-x_{n-1}\|^2\le F(z_n)-F(z_{n+1}).$$
\item The corresponding relative error (H2) in \cite{att-b-sv2013} is
$$\|\n F(z_{n+1})\|\le b\|x_{n+1}-x_{n}\|$$
consequently, in some sense, our condition may have a larger relative error.
 In \cite{FGP} the condition (H2) has the form
 $$\|\n F(z_{n+1})\|\le b_n\|x_{n+1}-x_{n}\|+c_n,\mbox{ where }b_n>0,\,c_n\ge 0.$$
 Moreover, in \cite{OCBP} is considered $(z_n)_{n\in\N}=(x_n,x_{n-1})_{n\in\N}$, hence their condition (H2) has the form
$$\|\n F(x_{n+1},x_n)\|\le b(\|x_{n+1}-x_{n}\|+\|x_n-x_{n-1}\|).$$
\item Further, since  in \cite{att-b-sv2013} and \cite{OCBP} $F$ is assumed to be lower semicontinuous  only, their condition (H3) has the form:
 there exists a subsequence $(z_{n_j} )_{j\in\N}$ of $(z_n)_{n\in\N}$
 such that
$z_{n_j}\To z^*$ and $F(z_{n_j})\To F(z^*),$ as $j\To+\infty.$
Of course in our case  this condition holds whenever $\omega((z_n)_{n\in\N})$ is nonempty since $F$ is continuous. In \cite{FGP} condition (H3) refers to some properties of the sequences $(a_{n\in\N}),(b_{n\in\N})$ and $(c_n)_{n\in\N}.$
\end{itemize}

Consequently, at least in the smooth setting, our abstract convergence result stated in Lemma \ref{abstrconv} below is an extension of the corresponding result in \cite{att-b-sv2013}, \cite{FGP} and \cite{OCBP}.
\end{remark}

\begin{lemma}\label{abstrconv}
Let $F : \R^m\times\R^m\To\R$ be a continuously Fr\'echet differentiable  function which satisfies the Kurdyka-{\L}ojasiewicz property at some point $z^*=(x^*,x^*)\in \R^m\times\R^m.$

Let us denote by $U$, $\eta$ and $\varphi: [0, \eta)\To\R_+$ the objects appearing in the definition of the
KL property at $z^*.$ Let $\sigma>\rho > 0$ be such that $B(z^*,\sigma)\subseteq U.$ Furthermore, consider the sequences $(x_n)_{n\in\N},(v_n)_{n\in\N},(w_n)_{n\in\N}$ and let $(z_n)_{n\in\N}=(v_n,w_n)_{n\in\N}\subseteq\R^m\times\R^m$  be a sequence that satisfies the conditions (H1), (H2), and (H3).

Assume  further that
\begin{equation}\label{e00}
\forall n\in\N : z_n \in B(z^*,\rho)\implies z_{n+1}\in B(z^*,\sigma)\mbox{ with }F(z_{n+1})\ge F(z^*).
\end{equation}
Moreover, the initial point $z_0$ is such that $z_0\in B(z^*,\rho),$  $F(z^*)\le F(z_0)<F(z^*)+\eta$ and
\begin{equation}\label{e01}
\|x^*-x_0\|+2\sqrt{\frac{F(z_0)-F(z^*)}{a}}+\frac{9b}{4a}\varphi(F(z_0)-F(z^*))<\frac{\rho}{c_1+c_2}.
\end{equation}

Then, the following statements hold.

One has that
$ z_n\in B(z^*,\rho)$ for all $n \in\ N.$
Further, $\sum_{n=1}^{+\infty}\|x_n-x_{n-1}\|<+\infty$ and the sequence $(x_n)_{n\in\N}$ converges to a point $\ol x\in\R^m.$ The sequence $(z_n)_{n\in\N}$ converges to  $\ol z= (\ol x,\ol x),$ moreover, we have $\ol z\in B(z^*,\sigma)\cap \crit (F)$ and $F(z_n)\To F(\ol z)=F(z^*),\, n \To+\infty.$
 \end{lemma}
Due to the technical details of the proof of Lemma \ref{abstrconv}, we will first present a sketch
of it in order to give a better insight.

 1. At first, our aim is to show by classical induction that $z_k\in B(z^*,\rho),\,F(z_k)<F(z^*)+\eta$ and the inequality $$2\|x_{k+1}-x_{k}\|\le\|x_{k}-x_{k-1}\|+\frac{9b}{4a}(\varphi(F(z_k)-F(z^*))-\varphi(F(z_{k+1})-F(z^*))$$
 holds, for every $k\ge 1.$ %We prove our claims in 4 steps.

 To this end we show that the assumptions in the hypotheses of Lemma \ref{abstrconv} assures that  $z_1\in B(z^*,\rho)$ and $F(z_1)<F(z^*)+\eta.$

  Further, we show that if $z_k\in B(z^*,\rho),\,F(z_k)<F(z^*)+\eta$ for some $k\ge 1,$ then $$2\|x_{k+1}-x_{k}\|\le\|x_{k}-x_{k-1}\|+\frac{9b}{4a}(\varphi(F(z_k)-F(z^*))-\varphi(F(z_{k+1})-F(z^*))),$$
 which combined with the previous step assures that the base case, $k=1$, in our induction process holds.

  Next, we take the inductive step and show that the previous statement  holds for every $k\ge 1.$

 2. By summing up   the inequality obtained at 1. from $k=1$ to $k=n$ and letting $n\To+\infty$ we obtain that the sequence $(x_n)_{n\in\N}$ is convergent and from here the conclusion of the Lemma easily follows.

 We now pass to a detailed presentation of this proof.
 \begin{proof}
We divide the proof into the following steps.

{\bf Step I.} We show that $z_1\in B(z^*,\rho)$  and $F(z_1)< F(z^*)+\eta.$

Indeed, $z_0\in B(z^*,\rho)$ and \eqref{e00} assures that $F(z_1)\ge F(z^*).$
 Further, (H1) assures that $\|x_1-x_0\|\le\sqrt{\frac{F(z_0)-F(z_1)}{a}}$ and since $\|x_1-x^*\|=\|(x_1-x_0)+(x_0-x^*)\|\le \|x_1-x_0\|+\|x_0-x^*\|$  and
 $F(z_1)\ge F(z^*)$  the condition  \eqref{e01} leads to
  $$\|x_1-x^*\|\le \|x_0-x^*\|+ \sqrt{\frac{F(z_0)-F(z^*)}{a}}<\frac{\rho}{c_1+c_2}.$$
 Now, from (H3) we have $\|z_1-z^*\|\le c_1\|x_1-x^*\|+c_2\|x_{0}-x^*\|$ hence
  $$\|z_1-z^*\|< c_1\frac{\rho}{c_1+c_2}+c_2\frac{\rho}{c_1+c_2}= \rho.$$
  Thus, $z_1\in B(z^*,\rho),$ moreover \eqref{e00} and (H1) provide that $F(z^*)\le F(z_2)\le F(z_1)\le F(z_0)< F(z^*)+\eta.$

{\bf Step II.} Next we show that whenever for a $k\ge 1$ one has $z_k\in B(z^*,\rho),\,F(z_k)<F(z^*)+\eta$ then it holds that
\begin{equation}\label{fontos}
2\|x_{k+1}-x_{k}\|\le\|x_{k}-x_{k-1}\|+\frac{9b}{4a}(\varphi(F(z_k)-F(z^*))-\varphi(F(z_{k+1})-F(z^*)).
\end{equation}
Hence, let $k\ge 1$ and assume that  $z_k\in B(z^*,\rho),\,F(z_k)<F(z^*)+\eta$. Note that from (H1) and \eqref{e00} one has $F(z^*)\le F(z_{k+1})\le F(z_k)<F(z^*)+\eta,$ hence $$F(z_k)-F(z^*),F(z_{k+1})-F(z^*)\in[0,\eta),$$ thus \eqref{fontos} is well stated. Now, if $x_k=x_{k+1}$ then \eqref{fontos} trivially holds.

Otherwise, from (H1) and \eqref{e00} one has
\begin{equation}\label{inter1}
F(z^*)\le F(z_{k+1})<F(z_k)<F(z^*)+\eta.
\end{equation}
Consequently,  $z_k\in B(z^*,\rho)\cap  \{z\in\R^n: F(z^*)<F(z)<F(z^*)+\eta\}$ and by using the KL inequality  we get
$$\varphi'(F(z_k)-F(z^*))\|\n F(z_k)\|\geq 1.$$
Since $\varphi$ is concave, and  \eqref{inter1} assures that $F(z_{k+1})-F(z^*)\in[0,\eta),$ one has
$$\varphi(F(z_k)-F(z^*))-\varphi(F(z_{k+1})-F(z^*))\ge\varphi'(F(z_k)-F(z^*))(F(z_k)-F(z_{k+1})),$$
consequently,
$$\varphi(F(z_k)-F(z^*))-\varphi(F(z_{k+1})-F(z^*))\ge\frac{F(z_k)-F(z_{k+1})}{\|\n F(z_k)\|}.$$
Now, by  using (H1) and (H2) we get  that
$$\varphi(F(z_k)-F(z^*))-\varphi(F(z_{k+1})-F(z^*))\ge\frac{a\|x_{k+1}-x_{k}\|^2}{b(\|x_{k+1}-x_{k}\|+\|x_k-x_{k-1}\|)}.$$
Consequently,
$$\|x_{k+1}-x_{k}\|\le\sqrt{\frac{b}{a}\left(\varphi(F(z_k)-F(z^*))-\varphi(F(z_{k+1})-F(z^*))\right)(\|x_{k+1}-x_{k}\|+\|x_k-x_{k-1}\|)}$$
and by arithmetical-geometrical mean inequality we have
\begin{align}\nonumber\|x_{k+1}-x_{k}\|&\le\frac{\|x_{k+1}-x_{k}\|+\|x_k-x_{k-1}\|}{3}\\
\nonumber&+\frac{3b}{4a}(\varphi(F(z_k)-F(z^*))-\varphi(F(z_{k+1})-F(z^*))),
\end{align}
which leads to \eqref{fontos}, that is
$$2\|x_{k+1}-x_{k}\|\le\|x_{k}-x_{k-1}\|+\frac{9b}{4a}(\varphi(F(z_k)-F(z^*))-\varphi(F(z_{k+1})-F(z^*)).$$

{\bf Step III.} Now we show by induction that \eqref{fontos} holds for every $k\ge 1.$  Indeed, Step II. can be applied for $k=1$ since according to Step I. $z_1\in B(z^*,\rho)$  and $F(z_1)< F(z^*)+\eta.$ Consequently, for $k=1$ the inequality \eqref{fontos} holds.

 Assume that \eqref{fontos} holds for every $k\in\{1,2,...,n\}$ and we show also that \eqref{fontos} holds for $k=n+1.$ Arguing as at Step II., the condition (H1) and \eqref{e00} assure that $F(z^*)\le F(z_{n+1})\le F(z_n)<F(z^*)+\eta,$ hence it remains to show that $z_{n+1}\in B(z^*,\rho).$
 By using the triangle inequality and (H3) one has
 \begin{align}\label{fonti}
 \|z_{n+1}-z^*\|&\le c_1\|x_{n+1}-x^*\|+c_2\|x_n-x^*\|\\
 \nonumber&=c_1\|(x_{n+1}-x_n)+(x_n-x_{n-1})+\cdots+(x_0-x^*)\|\\
 \nonumber&+c_2\|(x_{n}-x_{n-1})+(x_{n-1}-x_{n-2})+\cdots+(x_0-x^*)\|\\
\nonumber&\le c_1\|x_{n+1}-x_n\|+(c_1+c_2)\|x_0-x^*\|+(c_1+c_2)\sum_{k=1}^{n}\|x_{k}-x_{k-1}\|.
\end{align}

 By summing up \eqref{fontos} from $k=1$ to $k=n$ we obtain
 \begin{equation}\label{fontos1}
\sum_{k=1}^n \|x_k-x_{k-1}\|\le 2\|x_{1}-x_{0}\|-2\|x_{n+1}-x_{n}\|+\frac{9b}{4a}(\varphi(F(z_1)-F(z^*))-\varphi(F(z_{n+1})-F(z^*)).
\end{equation}

Combining \eqref{fonti} and \eqref{fontos1} and neglecting the negative terms we get
$$\|z_{n+1}-z^*\|\le (2c_1+2c_2)\|x_{1}-x_0\|+(c_1+c_2)\|x_0-x^*\|+(c_1+c_2)\frac{9b}{4a}\varphi(F(z_1)-F(z^*)).$$

But $\varphi$ is strictly increasing and $F(z_1)-F(z^*)\le F(z_0)-F(z^*)$, hence
$$\|z_{n+1}-z^*\|\le (2c_1+2c_2)\|x_{1}-x_0\|+(c_1+c_2)\|x_0-x^*\|+(c_1+c_2)\frac{9b}{4a}\varphi(F(z_0)-F(z^*)).$$

According to (H1) one has
\begin{equation}\label{x1}
\|x_{1}-x_0\|\le\sqrt{\frac{F(z_{0})-F(z_1)}{a}}\le \sqrt{\frac{F(z_{0})-F(z^*)}{a}},
\end{equation}
hence, from \eqref{e01} we get
$$\|z_{n+1}-z^*\|\le (c_1+c_2)\left(\|x_0-x^*\|+2\sqrt{\frac{F(z_{0})-F(z^*)}{a}}+\frac{9b}{4a}\varphi(F(z_0)-F(z^*))\right)<\rho.$$

Hence, we have shown so far that $z_n\in B(z^*,\rho)$ for all $n\in\N.$

{\bf Step IV.} According to Step III. the relation \eqref{fontos} holds for every $k\ge 1.$ But this implies that \eqref{fontos1} holds for every $n\ge 1.$ By using \eqref{x1}  and neglecting the nonpositive terms, \eqref{fontos1} becomes
\begin{equation}\label{fontos2}
\sum_{k=1}^n \|x_k-x_{k-1}\|\le 2\sqrt{\frac{F(z_{0})-F(z^*)}{a}}+\frac{9b}{4a}\varphi(F(z_1)-F(z^*)).
\end{equation}
Now letting $n\To+\infty$ in \eqref{fontos2} we obtain that
$$\sum_{k=1}^{\infty} \|x_k-x_{k-1}\|<+\infty.$$

Obviously the sequence $S_n=\sum_{k=1}^n\|x_k-x_{k-1}\|$ is Cauchy, hence, for all $\e>0$ there exists $N_{\e}\in \N$   such that for all $n\ge N_\e$ and for all $p\in \N$ one has
$$S_{n+p}-S_n\le \e.$$
But
$$S_{n+p}-S_n=\sum_{k={n+1}}^{n+p}\|x_k-x_{k-1}\|\ge \left\|\sum_{k={n+1}}^{n+p}(x_k-x_{k-1})\right\|=\|x_{n+p}-x_n\|$$
hence the sequence $(x_n)_{n\in\N}$ is Cauchy, consequently is convergent. Let
$$\lim_{n\To+\infty}x_n=\ol x.$$

Let $\ol z=(\ol x,\ol x).$ Now, from (H3) we have
$$\lim_{n\To+\infty}\|z_n-\ol z\|\le \lim_{n\To+\infty}(c_1\|x_n-\ol x\|+c_2\|x_{n-1}-\ol x\|)=0,$$
consequently $(z_n)_{n\in\N}$ converges to $\ol z.$

Further, $(z_n)_{n\in\N}\subseteq B(z^*,\rho)$ and $\rho<\sigma$, hence $\ol z\in B(z^*,\sigma).$
Moreover, from (H2) we have
$$\|\n F(\ol z)\|=\lim_{n\To+\infty}\|\n F(z_n)\|\le\lim_{n\To+\infty}b(\|x_{n+1}-x_n\|+\|x_n-x_{n-1}\|)=0$$
which shows that $\ol z\in\crit(F).$ Consequently, $\ol z\in B(z^*,\sigma)\cap \crit (F)$.

Finally, since $z_n\To \ol z,\, n\To+\infty$ an $F$ is continuous it is obvious that $\lim_{n\To+\infty} F(z_n)=F(\ol z).$ Further, since $F(z^*)\le F(z_n)<F(z^*)+\eta$ for all $n\ge 1$ and the sequence $(F(z_n))_{n\ge 1}$ is decreasing, obviously $F(z^*)\le F(\ol z)<F(z^*)+\eta.$ Assume that $F(z^*)< F(\ol z).$ Then, one has
$$\ol z\in B(z^*,\sigma)\cap  \{z\in\R^n: F(z^*)<F(z)<F(z^*)+\eta\}$$ and by using the KL inequality  we get
$$\varphi'(F(\ol z)-F(z^*))\|\n F(\ol z)\|\geq 1,$$
impossible since $\|\n F(\ol z)\|=0.$ Consequently $F(\ol z)=F(z^*).$
 \end{proof}
\begin{remark} One can observe that our conditions in Lemma \ref{abstrconv} are slightly different  to those in \cite{att-b-sv2013} and \cite{OCBP}. Indeed, we must assume that $z_0\in B(z^*,\rho)$ and  in the  right hand side of \eqref{e01} we have $\frac{\rho}{c_1+c_2}.$

Though does not fit into the framework of this paper, we are confident that  Lemma \ref{abstrconv} can be extended to the case  when we do not assume that $F$ is continuously Fr\'echet differentiable but only that $F$ is proper and lower semicontinuous. Then, the gradient of $F$ can be replaced by the limiting subdifferential of $F$. These assumptions will imply some slight modifications in the conclusion of Lemma \ref{abstrconv} and only the lines of proof at Step IV. must be substantially modified.
\end{remark}
%\begin{comment}
\begin{corollary}\label{corabstrconv} Assume that the sequences from the definition of $(z_n)_{n\in N}$ satisfy $v_n=x_n+\a_n(x_n-x_{n-1})$ and $w_n=x_n+\b_n(x_n-x_{n-1})$  for all $n\ge 1$,  where $(\a_n)_{n\in\N},(\b_n)_{n\in\N}$ are  bounded sequences. Let $c=sup_{n\in\N}(|\a_n|+|\b_n|).$ Then (H3) holds with $c_1=2+c$ and $c_2=c$. Further, Lemma \ref{abstrconv} holds true if we replace \eqref{e00} in its hypotheses by
 $$\eta < \frac{a(\sigma-\rho)^2}{4(1+c)^2}\mbox{ and }F(z_n)\ge F(z^*),\mbox{ for all }n\in\N,\,n\ge 1.$$
 \end{corollary}
 \begin{proof} The claim that (H3) holds with $c_1=2+c$ and $c_2=c$ is an easy verification. We have to show that \eqref{e00} holds, that is, $z_n\in B(z^*,\rho)$ implies $z_{n+1}\in B(z^*,\sigma)$ for all $n\in\N.$

According to (H1), the assumption that $F(z_n)\ge F(z^*)$ for all $n\ge 1$ and  the hypotheses of Lemma \ref{abstrconv},  we have
$$\|x_n-x_{n-1}\|\le\sqrt{\frac{F(z_{n-1})-F(z_{n})}{a}}\le\sqrt{\frac{F(z_0)-F(z_{n})}{a}}\le\sqrt{\frac{F(z_0)-F(z^*)}{a}}<\sqrt{\frac{\eta}{a}}$$
and
$$\|x_{n+1}-x_{n}\|\le\sqrt{\frac{F(z_{n})-F(z_{n+1})}{a}}\le\sqrt{\frac{F(z_0)-F(z_{n+1})}{a}}\le\sqrt{\frac{F(z_0)-F(z^*)}{a}}<\sqrt{\frac{\eta}{a}}$$
for all $n\ge 1$.

Assume now that $n\ge 1$ and $z_n\in B(z^*,\rho).$
Then, by using the triangle inequality we get
$$\|z_{n+1}-z^*\|=\|(z_{n+1}-z_n)+(z_n-z^*)\|\le \|z_{n+1}-z_n\|+\|z_n-z^*\|\le\|z_{n+1}-z_n\|+\rho.$$

Further,
\begin{align}\nonumber\|z_{n+1}-z_n\|&=\|(v_{n+1}-v_n,w_{n+1}-w_n)\|\\
\nonumber&\le\|x_{n+1}+\a_{n+1}(x_{n+1}-x_{n})-x_n-\a_n(x_n-x_{n-1})\|\\
\nonumber&+\|x_{n+1}+\b_{n+1}(x_{n+1}-x_{n})-x_n-\b_n(x_n-x_{n-1})\|\\
\nonumber&\le(2+|\a_{n+1}|+|\b_{n+1}|)\|x_{n+1}-x_n\|+(|\a_n|+|\b_n|)\|x_n-x_{n-1}\|\\
\nonumber&\le(2+c)\|x_{n+1}-x_n\|+c\|x_{n}-x_{n-1}\|,
\end{align}
where $c=\sup_{n\in\N}(|\a_{n}|+|\b_{n}|).$

Consequently, we have
$$\|z_{n+1}-z^*\|\le (2+c)\|x_{n+1}-x_n\|+c\|x_{n}-x_{n-1}\|+\rho<(2+2c)\sqrt{\frac{\eta}{a}}+\rho\le\sigma,$$ which
is exactly $z_{n+1}\in B(z^*,\sigma).$ Further, arguing analogously as at Step I. in the proof of Lemma \ref{abstrconv}, we obtain that $z_{1}\in  B(z^*,\rho)\subseteq B(z^*,\sigma)$ and this concludes the proof.
\end{proof}
%\end{comment}

Now we are ready to formulate the following result.

\begin{theorem}\label{thabstrconv}(Convergence to a critical point). Let $F : \R^m\times\R^m\To\R$ be a continuously Fr\'echet differentiable function
and let $(z_n)_{n\in N}=(x_n+\a_n(x_n-x_{n-1}),x_n+\b_n(x_n-x_{n-1}))_{n\in\N}$  be a sequence that
satisfies (H1) and (H2), (with the convention $x_{-1}\in\R^m$), where $(\a_n)_{n\in\N},(\b_n)_{n\in\N}$ are bounded sequences. Moreover, assume that $\omega((z_n)_{n\in\N})$ is nonempty and that $F$ has the Kurdyka-{\L}ojasiewicz property at a point $z^*=(x^*,x^*)\in \omega((z_n)_{n\in\N}).$
Then, the sequence $(x_n)_{n\in\N}$ converges to $x^*$, $(z_n)_{n\in\N}$ converges to $z^*$ and $z^*\in\crit(F).$
\end{theorem}
\begin{proof}
We will apply Corollary \ref{corabstrconv}. Since $z^*=(x^*,x^*)\in \omega((z_n)_{n\in\N})$ there exists a subsequence
$(z_{n_k})_{k\in\N}$ such that
$$z_{n_k}\To z^*,\, k\To+\infty.$$
From (H1) we get that the sequence $(F(z_n))_{n\in\N}$ is decreasing and obviously $F(z_{n_k})\To F(z^*),\, k\To+\infty,$ which implies that
\begin{equation}\label{tce1}
F(z_n)\To F(z^*),\,n\To+\infty\mbox{ and }F(z_n)\ge F(z^*),\,\mbox{for all }n\in\N.
\end{equation}

We show next that $x_{n_k}\To x^*,\, k\To+\infty.$ Indeed, from (H1) one has
$$a\|x_{n_k}-x_{{n_k}-1}\|^2\le F(z_{n_k-1})-F(z_{{n_k}})$$
and obviously the right side of the above inequality goes to $0$ as $k\To+\infty.$
Hence,
$$\lim_{k\To+\infty}(x_{n_k}-x_{n_k-1})=0.$$
Further, since the sequences $(\a_n)_{n\in\N},(\b_n)_{n\in\N}$ are bounded we get
$$\lim_{k\To+\infty}\a_{n_k}(x_{n_k}-x_{n_k-1})=0$$
and
$$\lim_{k\To+\infty}\b_{n_k}(x_{n_k}-x_{n_k-1})=0.$$
Finally, $z_{n_k}\To z^*,\, k\To+\infty$ is equivalent to
$$x_{n_k}-x^*+\a_{n_k}(x_{n_k}-x_{{n_k}-1})\To 0,\, k\To+\infty$$
and
$$x_{n_k}-x^*+\b_{n_k}(x_{n_k}-x_{{n_k}-1})\To 0,\, k\To+\infty,$$
which lead to the desired conclusion, that is
\begin{equation}\label{tce2}
x_{n_k}\To x^*,\, k\To+\infty.
\end{equation}
The KL property around $z^*$ states the existence of quantities $\varphi$, $U$, and $\eta$ as in Definition \ref{KL-property}.
Let $\sigma > 0 $ be such that $ B(z^*, \sigma)\subseteq U$ and $\rho\in(0,\sigma).$  If necessary we shrink $\eta$ such that
$\eta < \frac{a(\sigma-\rho)^2}{4(1+c)^2},$ where $c=\sup_{n\in\N}(|\a_n|+|\b_n|).$

Now, since the functions $F$ and $\varphi$ are continuous and $F(z_n)\To F(z^*),\,n\To+\infty$, further $\varphi(0)=0$ and $z_{n_k}\To z^*,\,x_{n_k}\To x^*,\, k\To+\infty$ we conclude that there exists
$n_0\in\N,\,n_0\ge 1$ such that $z_{n_0}\in B(z^*,\rho)$ and $F(z^*)\le F(z_{n_0})<F(z^*)+\eta,$
moreover
$$\|x^*-x_{n_0}\|+2\sqrt{\frac{F(z_{n_0})-F(z^*)}{a}}+\frac{9b}{4a}\varphi(F(z_{n_0})-F(z^*))<\frac{\rho}{c_1+c_2}.$$

Hence, Corollary \ref{corabstrconv} and consequently Lemma \ref{abstrconv} can be applied to the sequence $(u_n)_{n\in\N},\, u_n=z_{n_0+n}.$

Thus, according to Lemma \ref{abstrconv}, $(u_n)_{n\in\N}$ converges to a point $(\ol x,\ol x)\in\crit(F),$ consequently  $(z_n)_{n\in\N}$ converges to $(\ol x,\ol x).$ But then, since $\omega((z_n)_{n\in\N})=\{(\ol x,\ol x)\}$ one has $x^*=\ol x.$  Hence,
$(x_n)_{n\in\N}$ converges to $x^*$, $(z_n)_{n\in\N}$ converges to $z^*$ and $z^*\in\crit(F).$
\end{proof}

\begin{remark} We emphasize that the main advantage of the abstract convergence results from this section is that can be applied also for algorithms where the the gradient of the objective is evaluated in iterations that contain the inertial therm. This is due to the fact that the sequence $(z_n)_{n\in\N}$
may have the form proposed in Corollary \ref{corabstrconv} and Theorem \ref{thabstrconv}.
\end{remark}

\subsection{The convergence of the numerical method \eqref{generaldiscrete}}

Based on the abstract convergence results obtained in the previous section, in this section we show the convergence of the sequences generated by Algorithm \eqref{generaldiscrete}. The main tool in our forthcoming analysis is the so called descent lemma, see \cite{Nest}, which in our setting reads as

%\begin{lemma}\label{desc} Let $g:\R^m\To\R$ be Fr\'echet differentiable with $L_g$ Lipschitz continuous gradient. Then
\begin{equation}\label{desc}g(y)\le g(x)+\<\n g(x),y-x\>+\frac{L_g}{2}\|y-x\|^2,\,\forall x,y\in\R^m.
\end{equation}
%\end{lemma}

Now we are able to obtain a decrease property for the iterates generated by \eqref{generaldiscrete}.

\begin{lemma}\label{decreasing} In the settings of problem \eqref{opt-pb}, for some starting points $x_0,x_{-1}\in\R^m,$ let $(x_n)_{n\in\N},\,(y_n)_{n\in\N}$ be the sequences generated by the numerical scheme \eqref{generaldiscrete}. Consider the sequences

$\ds A_{n-1}=\frac{2-s L_g}{2s}\left(\frac{(1+\b)n+\a}{n+\a}\right)^2-\frac{\b n((1+\b)n+\a)}{s(n+\a)^2},$

$\ds B_n=\frac{2-s L_g}{2s}\left(\frac{\b n}{n+\a}\right)^2,$

$\ds C_{n-1}=\frac{2-s L_g}{2s}\frac{\b n-\b}{n+\a-1}\frac{(1+\b)n+\a}{n+\a}-\frac{1}{2s}\frac{\b n-\b}{n+\a-1}\frac{\b n}{n+\a}$
and

$\ds \d_n=\frac12(A_{n-1}-C_{n-1}-B_n+C_n),$
for all $n\in\N,\,n\ge 1$.
\\Then, there exists $N\in\N$ such that
\begin{itemize}
\item[(i)] The sequence $ \left(g(y_{n})+\d_{n}\|x_{n}-x_{n-1}\|^2\right)_{n\ge N}$ is nonincreasing and $\d_n>0$ for all $n\ge N$.
\end{itemize}
Assume that $g$ is bounded from below. Then, the following statements hold.
\begin{itemize}
\item[(ii)] The sequence $\left(g(y_{n})+\d_{n}\|x_{n}-x_{n-1}\|^2\right)_{n\in \N}$ is convergent;
\item[(iii)] $\sum_{n\ge 1}\|x_n-x_{n-1}\|^2<+\infty.$
\end{itemize}
\end{lemma}
Due to the technical details of the proof of Lemma \ref{decreasing}, we will first present a sketch
of it in order to give a better insight.

1. We start from \eqref{generaldiscrete} and \eqref{desc} to obtain
$$g(y_{n+1})-\frac{L_g}{2}\|y_{n+1}-y_n\|^2\le g(y_n)+ \frac{1}{s}\<y_n-x_{n+1},y_{n+1}-y_n\>,\mbox{ for all }n\in\N.$$
From here, by using  equalities only, we obtain the key inequality
$$\frac{\Delta_{n+1}}{2}\|x_{n+1}-x_n\|^2+\frac{\Delta_n}{2}\|x_n-x_{n-1}\|^2\le (g(y_n)+\d_n\|x_n-x_{n-1}\|^2)- (g(y_{n+1})+\d_{n+1}\|x_{n+1}-x_n\|^2),$$
for all $n\in\N,\,n\ge 1,$ where $\Delta_n=A_{n-1}-C_{n-1}+B_n-C_n.$ Further, we show that $C_n,\Delta_n$ and $\d_n$ are positive after an index $N.$

2. Now, the key inequality emphasized at 1. implies at once (i), and if we assume that $g$ is bounded from below also (ii) and (iii) follows in a straightforward way.

We now pass to a detailed presentation of this proof.
\begin{proof}

From \eqref{generaldiscrete} we have $\n g(y_n)=\frac{1}{s}(y_n-x_{n+1})$, hence
$$\<\n g(y_n),y_{n+1}-y_n\>=\frac{1}{s}\<y_n-x_{n+1},y_{n+1}-y_n\>,\mbox{ for all }n\in\N.$$
Now, from \eqref{desc} we obtain
$$g(y_{n+1})\le g(y_n)+\<\n g(y_n),y_{n+1}-y_n\>+\frac{L_g}{2}\|y_{n+1}-y_n\|^2,$$
consequently we have
\begin{equation}\label{e1}
g(y_{n+1})-\frac{L_g}{2}\|y_{n+1}-y_n\|^2\le g(y_n)+ \frac{1}{s}\<y_n-x_{n+1},y_{n+1}-y_n\>,\mbox{ for all }n\in\N.
\end{equation}

Further, for all $n\in\N$ one has
$$\<y_n-x_{n+1},y_{n+1}-y_n\>=-\|y_{n+1}-y_n\|^2+\<y_{n+1}-x_{n+1},y_{n+1}-y_n\>,$$
 and
$$y_{n+1}-x_{n+1}=\frac{\b(n+1)}{n+\a+1}(x_{n+1}-x_n),$$
hence,
\begin{equation}\label{e2}
g(y_{n+1})+\left(\frac{1}{s}-\frac{L_g}{2}\right)\|y_{n+1}-y_n\|^2\le g(y_n)+ \frac{\frac{\b(n+1)}{n+\a+1}}{s}\<x_{n+1}-x_n,y_{n+1}-y_n\>.
\end{equation}

Since $$y_{n+1}-y_n=\frac{(1+\b)n+\a+\b+1}{n+\a+1}(x_{n+1}-x_n)-\frac{\b n}{n+\a}(x_n-x_{n-1}),$$
we have,
\begin{align}\nonumber\|y_{n+1}-y_n\|^2&=\left\|\frac{(1+\b)n+\a+\b+1}{n+\a+1}(x_{n+1}-x_n)-\frac{\b n}{n+\a}(x_n-x_{n-1})\right\|^2\\
\nonumber&=\left(\frac{(1+\b)n+\a+\b+1}{n+\a+1}\right)^2\|x_{n+1}-x_n\|^2+\left(\frac{\b n}{n+\a}\right)^2\|x_n-x_{n-1}\|^2\\
\nonumber&-2\frac{(1+\b)n+\a+\b+1}{n+\a+1}\frac{\b n}{n+\a}\<x_{n+1}-x_n,x_n-x_{n-1}\>,
\end{align}
and
\begin{align}\nonumber\<x_{n+1}-x_n,y_{n+1}-y_n\>&=\left\<x_{n+1}-x_n,\frac{(1+\b)n+\a+\b+1}{n+\a+1}(x_{n+1}-x_n)-\frac{\b n}{n+\a}(x_n-x_{n-1})\right\>\\
\nonumber&=\frac{(1+\b)n+\a+\b+1}{n+\a+1}\|x_{n+1}-x_n\|^2-\frac{\b n}{n+\a}\<x_{n+1}-x_n,x_n-x_{n-1}\>,
\end{align}
for all $n\in\N.$

Replacing the above equalities in \eqref{e2}, we obtain

\begin{align}\nonumber& g(y_{n+1})+\left(\frac{2-s L_g}{2s}\left(\frac{(1+\b)n+\a+\b+1}{n+\a+1}\right)^2-\frac{\b(n+1)((1+\b)n+\a+\b+1)}{s(n+\a+1)^2}\right)\|x_{n+1}-x_n\|^2\le\\
\nonumber&g(y_n)-\frac{2-s L_g}{2s}\left(\frac{\b n}{n+\a}\right)^2\|x_{n}-x_{n-1}\|^2+\\
\nonumber& \left(\frac{2-s L_g}{s}\frac{\b n}{n+\a}\frac{(1+\b)n+\a+\b+1}{n+\a+1}-\frac{1}{s}\frac{\b n}{n+\a}\frac{\b(n+1)}{n+\a+1}\right)\<x_{n+1}-x_n,x_n-x_{n-1}\>,
\end{align}
for all $n\in\N.$

 Hence, for every $n\in\N,\,n\ge1$, we have
 $$g(y_{n+1})+A_{n}\|x_{n+1}-x_n\|^2-2C_n\<x_{n+1}-x_n,x_n-x_{n-1}\>\le g(y_n)-B_n\|x_{n}-x_{n-1}\|^2.$$

By using the equality
\begin{equation}\label{e5}
-2\<x_{n+1}-x_n,x_n-x_{n-1}\>=\|x_{n+1}+x_{n-1}-2x_n\|^2-\|x_{n+1}-x_n\|^2-\|x_n-x_{n-1}\|^2
\end{equation}
we obtain
$$g(y_{n+1})+(A_n-C_n)\|x_{n+1}-x_n\|^2+C_n\|x_{n+1}+x_{n-1}-2x_n\|^2\le g(y_n)+(C_n-B_n)\|x_n-x_{n-1}\|^2,$$
for all $n\in\N,\,n\ge 1.$

Note that $\d_n=\frac12(A_{n-1}-C_{n-1}-B_n+C_n)$, hence we have
\begin{align}\nonumber& g(y_{n+1})+\d_{n+1}\|x_{n+1}-x_n\|^2+(A_n-C_n-\d_{n+1})\|x_{n+1}-x_n\|^2+C_n\|x_{n+1}+x_{n-1}-2x_n\|^2\le\\
\nonumber&g(y_n)+\d_n\|x_n-x_{n-1}\|^2+(C_n-B_n-\d_n)\|x_n-x_{n-1}\|^2,
\end{align}
for all $n\in\N,\, n\ge 1.$

Let us denote $\Delta_n=A_{n-1}-C_{n-1}+B_{n}-C_{n}$ for all $n\ge 1.$ Then, $$A_n-C_n-\d_{n+1}=\frac12(A_n-C_n+B_{n+1}-C_{n+1})=\frac{\Delta_{n+1}}{2}$$ and
$$C_n-B_n-\d_n=-\frac{\Delta_n}{2},$$
 for all $n\ge 1,$ consequently the following inequality holds.
\begin{align}\label{e6}
& C_n\|x_{n+1}+x_{n-1}-2x_n\|^2+\frac{\Delta_n}{2}\|x_n-x_{n-1}\|^2+\frac{\Delta_{n+1}}{2}\|x_{n+1}-x_n\|^2\le\\
\nonumber& (g(y_n)+\d_n\|x_n-x_{n-1}\|^2)- (g(y_{n+1})+\d_{n+1}\|x_{n+1}-x_n\|^2),
\end{align}
for all $n\in\N,\, n\ge 1.$

Since $0<\b<1$ and $s<\frac{2(1-\b)}{L_g}$  we have
\begin{align}
\nonumber&\lim_{n\To+\infty}A_n=\frac{(2-s L_g)(\b+1)^2-2\b-2\b^2}{2s}>0,\\
\nonumber&\lim_{n\To+\infty}B_n=\frac{(2-s L_g)\b^2}{2s}>0,\\
\nonumber&\lim_{n\To+\infty}C_n= \frac{(2-s L_g)(\b^2+\b)-\b^2}{2s} >0,\\
\nonumber&\lim_{n\To+\infty}\Delta_n=\frac{2-s L_g-2\b}{2s}>0, \mbox{ and}\\
\nonumber&\lim_{n\To+\infty}\d_n=\frac{2+2\b-2\b^2-s L_g(2\b+1)}{2s}>0.
\end{align}

Hence, there exists $N\in\N,\,N\ge 1$ and $C>0,\,D>0$ such that for all $n\ge N$ one has
$$C_n\ge C,\, \frac{\Delta_n}{2}\ge D\mbox{ and }\d_n>0$$ which, in the view of \eqref{e6}, shows (i), that is,  the sequence $g(y_n)+\d_n\|x_n-x_{n-1}\|^2$ is nonincreasing for $n\ge N.$

%Since $\lim_{n\To+\infty}\Delta_n>0$ there exists $D>0$ and $p\in\N$ such that $\Delta_n\ge D$ for all $n\ge p.$
 By using \eqref{e6} again, we obtain
 \begin{align}\label{e61}
 0&\le C\|x_{n+1}+x_{n-1}-2x_n\|^2+D\|x_n-x_{n-1}\|^2+D\|x_{n+1}-x_n\|^2\\
 \nonumber&\le(g(y_n)+\d_n\|x_n-x_{n-1}\|^2)- (g(y_{n+1})+\d_{n+1}\|x_{n+1}-x_n\|^2),
 \end{align}
for all $n\ge N,$ or more convenient, that
 \begin{equation}\label{ineq}
0\le D\|x_{n+1}-x_{n}\|^2\le (g(y_n)+\d_n\|x_n-x_{n-1}\|^2)- (g(y_{n+1})+\d_{n+1}\|x_{n+1}-x_n\|^2),
\end{equation}
for all $n\ge N.$ Let $r>N.$ By summing up the latter relation from $n=N$ to $n=r$ we get
$$
D\sum_{n=N}^r\|x_{n+1}-x_{n}\|^2\le (g(y_N)+\d_N\|x_N-x_{N-1}\|^2)-(g(y_{r+1})+\d_{r+1}\|x_{r+1}-x_r\|^2)
$$
which leads to
\begin{equation}\label{forcoercive}
g(y_{r+1})+D\sum_{n=N}^r\|x_{n+1}-x_{n}\|^2\le g(y_N)+\d_N\|x_N-x_{N-1}\|^2.
\end{equation}
Now,  if we assume that $g$ is bounded from below, by letting $r\To+\infty$  we obtain
$$\sum_{n=N}^{\infty}\|x_{n+1}-x_{n}\|^2<+\infty$$ which proves (iii).

The latter relation also shows that
$$\lim_{n\To+\infty}\|x_n-x_{n-1}\|^2=0,$$
hence $$\lim_{n\To+\infty}\d_n\|x_n-x_{n-1}\|^2=0.$$
But then, by using the  assumption that the function $g$ is bounded from below we obtain that the sequence $(g(y_n)+\d_n\|x_n-x_{n-1}\|^2)_{n\in\N}$ is bounded from below. On the other hand, from (i) we have that  the sequence $(g(y_n)+\d_n\|x_n-x_{n-1}\|^2)_{n\ge N}$ is nonincreasing, hence there exists
$$\lim_{n\To+\infty} g(y_n)+\d_n\|x_n-x_{n-1}\|^2\in\R.$$
\end{proof}

\begin{remark}\label{r1} Observe that conclusion (iii) in Lemma \ref{decreasing} assures that the sequence  $(x_n-x_{n-1})_{n\in\N}\in l^2,$ in particular that
\begin{equation}\label{e7}
\lim_{n\To+\infty}(x_n-x_{n-1})=0.
\end{equation}
Note that according the proof of  Lemma \ref{decreasing}, one has $\d_n\|x_n-x_{n-1}\|^2\To 0,\,n\To+\infty.$ Thus, (ii) assures that there exists the limit $\lim_{n\To+\infty} g(y_n)\in\R.$
\end{remark}

In what follows, in order to apply our abstract convergence result obtained at Theorem \ref{thabstrconv}, we introduce a function and a sequence that will play the role of the function $F$ and the sequence $(z_n)$ studied in the previous section. Consider  the sequence
$$u_n=\sqrt{2\d_n}(x_{n}-x_{n-1})+y_n\mbox{, for all }n\in\N,\,n\ge N$$ and the sequence $z_n=(y_{n+N},u_{n+N})$ for all $n\in\N,$
where $N$ and $\d_n$ were defined in Lemma \ref{decreasing}. Let us introduce the following notations:
 $$\tilde{x}_n=x_{n+N}\mbox{ and }\tilde{y}_n=y_{n+N},$$
$$\a_n=\frac{\b (n+N)}{n+N+\a}\mbox{ and }\b_n=\sqrt{2\d_{n+N}}+\frac{\b (n+N)}{n+N+\a},$$
  for all $n\in\N.$
  Then obviously the sequences $(\a_n)_{n\in\N}$ and $(\b_n)_{n\in\N}$ are bounded, (actually they are convergent), and for each $n\in\N$, the sequence $z_n$ has the form
  \begin{equation}\label{zform}
z_n=\left(\tilde{x}_{n}+\a_n(\tilde{x}_{n}-\tilde{x}_{n-1}),\tilde{x}_{n}+\b_n(\tilde{x}_{n}-\tilde{x}_{n-1})\right).
\end{equation}

Consider further the following regularization of $g$
$$H:\R^m\times\R^m\To\R,\,H(x,y)=g(x)+\frac12\|y-x\|^2.$$

Then, for every $n\in\N$ one has
$$H(z_n)=g(\tilde{y}_n)+\d_{n+N}\|\tilde{x}_n-\tilde{x}_{n-1}\|^2.$$

Now, \eqref{ineq} becomes
\begin{equation}\label{H1forH}
D\|\tilde{x}_{n+1}-\tilde{x}_{n}\|\le H(z_n)-H(z_{n+1}),\,\mbox{ for all }n\in\N,
\end{equation}
which is exactly our condition (H1) applied to the function $H$ and the sequences $(\tilde{x}_n)_{n\in\N}$ and $(z_n)_{n\in\N}.$

\begin{remark}\label{discussenergy}
We emphasize that $H$ is strongly related to the total energy of the continuous dynamical systems \eqref{ee11} and \eqref{eee11}, (for other works where a similar regularization has been used we refer to \cite{BCL1,BCL,OCBP}).
Indeed, the total energy of the systems \eqref{ee11} and  \eqref{eee11} (see \cite{att-c-p-r-math-pr2018,APR}), is given by
$E:[t_0,+\infty)\to\R,\,E(t)=g(x(t))+\frac12\|\dot{x}(t)\|^2.$
Then, the explicit discretization of $E$ (see \cite{APR}), leads to
$E_n=g(x_n)+\frac12\|x_n-x_{n-1}\|^2=H(x_n,x_{n-1})$
and this fact was thoroughly exploited in \cite{OCBP}.
 However, observe that $H(z_n)$ cannot be obtained via the usual implicit/explicit discretization of $E.$
Nevertheless, $H(z_n)$ can be obtained from a discretization of $E$ by using the method presented in \cite{att-c-p-r-math-pr2018} which suggest to discretize $E$ in the form
$$E_n=g(\mu_n)+\frac12\|\eta_n\|^2,$$
where $\mu_n$ and $\eta_n$ are linear combinations of $x_n$ and $x_{n-1}.$ In our case we take $\mu_n=\tilde{y}_n$ and $\eta_n=\sqrt{2\d_{n+N}}(\tilde{x}_n-\tilde{x}_{n-1})$ and we obtain
$$E_n=g(\tilde{y}_n)+\d_{n+N}\|\tilde{x}_n-\tilde{x}_{n-1}\|^2=H(z_n).$$

\begin{comment}
However, if we consider  the energy functional
$$E_1(t)=g\left(\eta\dot{x}(t)+x(t)\right)+\frac12\left(1+{\eta\g}+\frac{\a\eta}{t}\right)\|\dot{x}(t)\|^2,\,\eta>0$$
introduced in  \cite{BCL-AA} in connection to the continuous dynamical system \eqref{eee11}, then the implicit discretization of $E_1$ leads to something very similar to $H(z_n).$ Indeed, for the implicit discretization of $E_1$, take the variable stepsize $h_n,$ and  consider $x_n = x(t_n).$
 Then the implicit discretization of $E_1$ is
 $$E_{1,n}=g\left(x_n+\frac{\eta}{h_n}(x_n-x_{n-1})\right)+\nu_n\|x_n-x_{n-1}\|^2,$$

\end{comment}
\end{remark}

The fact that $H$  and the sequences $(\tilde{x}_n)_{n\in\N}$ and $(z_n)_{n\in\N}$ are satisfying also condition (H2) is underlined in Lemma \ref{H2forH} (ii).

\begin{lemma}\label{H2forH} Consider the function $H$ and the sequences $(\tilde{x}_n)_{n\in\N}$ and $(z_n)_{n\in\N}$ defined above. Then, the following statements hold true.
\begin{itemize}
\item[(i)] $\crit(H)=\{(x,x)\in\R^m\times\R^m:x\in\crit(g)\}$;
\item[(ii)] There exists $b>0$ such that $\|\n H(z_n)\|\le b(\|\tilde{x}_{n+1}-\tilde{x}_n\|+\|\tilde{x}_{n}-\tilde{x}_{n-1}\|),$ for all $n\in\N$.
\end{itemize}
\end{lemma}
\begin{proof}
For (i) observe that  $\n H(x,y)=(\n g(x)+x-y,y-x)$, hence, $\n H(x,y)=(0,0)$ leads to
$x=y$ and $\n g(x)=0.$
Consequently
$$\crit(H)=\{(x,x)\in\R^m\times\R^m: x\in\crit(g)\}.$$

(ii) By using  \eqref{generaldiscrete}, for every $n\in \N,\,n\ge N$ we have
\begin{align}\nonumber\|\n H(y_n,u_n)\|&=\sqrt{\|\n g(y_n)+y_n-u_n\|^2+\|u_n-y_n\|^2}\le\sqrt{2\|\n g(y_n)\|^2+2\|y_n-u_n\|^2+\|u_n-y_n\|^2}\\
\nonumber&=\sqrt{2\|\n g(y_n)\|^2+6\d_n\|x_n-x_{n-1}\|^2}\le\sqrt{2}\|\n g(y_n)\|+\sqrt{6\d_n}\|x_{n}-x_{n-1}\|\\
\nonumber&=\frac{\sqrt{2}}{s}\left\|\left(x_n+\frac{\b n}{n+\a}(x_n-x_{n-1})\right)-x_{n+1}\right\|+\sqrt{6\d_n}\|x_{n}-x_{n-1}\|\\
\nonumber&\le\frac{\sqrt{2}}{s}\|x_{n+1}-x_n\|+\left(\frac{\sqrt{2}\b  n}{s(n+\a)}+\sqrt{6\d_n}\right)\|x_{n}-x_{n-1}\|.
\end{align}
Let $b=\max\left\{\frac{\sqrt{2}}{s},\sup_{n\ge N}\left(\frac{\sqrt{2}\b  n}{s(n+\a)}+\sqrt{6\d_n}\right)\right\}.$
Then, obviously $b>0$ and for all $n\in\N$ it holds
$$\|\n H(z_n)\|\le b(\|\tilde{x}_{n+1}-\tilde{x}_n\|+\|\tilde{x}_{n}-\tilde{x}_{n-1}\|).$$
\end{proof}

\begin{remark} Till now we did not take any advantage from the conclusions (ii) and (iii) of Lemma \ref{decreasing}. In the next result we show that under the assumption that $g$ is bounded from below, the limit sets $\omega((x_n)_{n\in\N})$ and $\omega((z_n)_{n\in\N})$ are strongly connected. This connection is due to the fact that in case $g$ is bounded from below then \eqref{e7} holds, that is, one has $\lim_{n\To+\infty}(x_n-x_{n-1})=0$.

 Moreover, we emphasize some useful properties of the regularization $H$ which occur  when we assume that $g$ is bounded from below.
\end{remark}
In the following result we use the distance function to a set, defined for $A\subseteq\R^m$ as $$\dist(x,A)=\inf_{y\in A}\|x-y\|\mbox{ for all }x\in\R^m.$$

\begin{lemma}\label{regularization} In the settings of problem \eqref{opt-pb}, for some starting points $x_0=x_{-1}\in\R^m,$ consider the sequences $(x_n)_{n\in\N},\,(y_n)_{n\in\N}$ generated by Algorithm \eqref{generaldiscrete}. Assume that $g$ is bounded from below. Then, the following statements hold true.
\begin{itemize}
\item[(i)] $\omega((u_n)_{n\in\N})=\omega((y_n)_{n\in\N})=\omega((x_n)_{n\in\N})\subseteq\crit(g)$, further $\omega((z_n)_{n\in\N})\subseteq\crit(H)$ and $\omega((z_n)_{n\in\N})=\{(\ol x,\ol x)\in\R^m\times\R^m: \ol x\in\omega((x_n)_{n\in\N})\}$;
\item[(ii)] $(H(z_n))_{n\in\N}$ is convergent and $H$ is constant on $\omega((z_n)_{n\in\N})$;
\item[(iii)]  $\|\n H(y_n,u_n)\|^2\le \frac{2}{s^2}\|x_{n+1}-x_n\|^2+2\left(\left(\frac{\b n}{s(n+\a)}-\sqrt{2\d_n}\right)^2+\d_n\right)\|x_{n}-x_{n-1}\|^2$ for all $n\in\N,\,n\ge N$.
\end{itemize}

Assume that $(x_n)_{n\in\N}$ is bounded. Then,
\begin{itemize}
\item[(iv)] $\omega((z_n)_{n\in\N})$ is nonempty and compact;
\item[(v)] $\lim_{n\To+\infty} \dist(z_n,\omega((z_n)_{n\in\N}))=0.$
\end{itemize}
\end{lemma}
\begin{proof}
(i) Let $\ol x\in\omega((x_n)_{n\in\N}).$ Then, there exists a subsequence $(x_{n_k})_{k\in \N}$ of $(x_n)_{n\in\N}$ such that
$$\lim_{k\to+\infty}x_{n_k}=\ol x.$$
Since by \eqref{e7} $\lim_{n\To+\infty}(x_n-x_{n-1})=0$ and the sequences $(\sqrt{2\d_n})_{n\in\N},\,\left(\frac{\b n}{n+\a}\right)_{n\in\N}$ converge, we obtain that
$$\lim_{k\to+\infty}y_{n_k}=\lim_{k\to+\infty}u_{n_k}=\lim_{k\to+\infty}x_{n_k}=\ol x,$$
which shows that $$\omega((x_n)_{n\in\N})\subseteq \omega((u_n)_{n\in\N})\mbox{ and }\omega((x_n)_{n\in\N})\subseteq \omega((y_n)_{n\in\N}).$$
Further, from \eqref{generaldiscrete}, the continuity of $\n g$ and \eqref{e7}, we obtain that
\begin{align}\nonumber\n g(\ol x)&=\lim_{k\To +\infty} \n g(y_{n_k})=\frac{1}{s}\lim_{k\To +\infty}(y_{n_k}-x_{n_k+1})\\
\nonumber&=\frac{1}{s}\lim_{k\To +\infty}\left[(x_{n_k}-x_{n_k+1})+\frac{\b n_k}{n_k+\a}(x_{n_k}-x_{n_k-1})\right]=0.
\end{align}
Hence, $\omega((x_n)_{n\in\N})\subseteq\crit(g).$ Conversely, if $\ol y\in \omega((y_n)_{n\in\N})$ then, from \eqref{e7} results that $\ol y\in \omega((x_n)_{n\in\N}).$ Further, if $\ol u\in \omega((u_n)_{n\in\N})$ then by using \eqref{e7} again we obtain that  $\ol u\in \omega((y_n)_{n\in\N}).$
Hence,
$$\omega((y_n)_{n\in\N})=\omega((u_n)_{n\in\N})=\omega((x_n)_{n\in\N})\subseteq\crit(g).$$

Obviously $\omega((\tilde{x}_n)_{n\in\N})=\omega((x_n)_{n\in\N})$ and since the sequences $(\a_n)_{n\in\N},\,(\b_n)_{n\in\N}$ are bounded, (convergent), from \eqref{e7} one gets
\begin{equation}\label{rege1}
\lim_{n\To+\infty}\a_n(\tilde{x}_{n}-\tilde{x}_{n-1})=\lim_{n\To+\infty}\b_n(\tilde{x}_{n}-\tilde{x}_{n-1})=0.
\end{equation}
Let $(\ol x,\ol y)\in \omega((z_n)_{n\in\N}).$ Then, there exists a subsequence $(z_{n_k})_{k\in\N}$ such that
$z_{n_k}\To(\ol x,\ol y),\,k\To+\infty.$
But we have
$z_n=\left(\tilde{x}_{n}+\a_n(\tilde{x}_{n}-\tilde{x}_{n-1}),\tilde{x}_{n}+\b_n(\tilde{x}_{n}-\tilde{x}_{n-1})\right),$ for all $n\in\N$, consequently from \eqref{rege1} we obtain
$$\tilde{x}_{n_k}\To \ol x\mbox{ and }\tilde{x}_{n_k}\To \ol y,\,k\To+\infty.$$
Hence, $\ol x=\ol y$ and $\ol x\in \omega((x_n)_{n\in\N})$  which shows that
$$\omega((z_n)_{n\in\N})\subseteq\{(\ol x,\ol x)\in\R^m\times\R^m: \ol x\in\omega((x_n)_{n\in\N})\}.$$
Conversely, if $\ol x \in \omega((\tilde{x}_n)_{n\in\N})$ then there exists a subsequence $(\tilde{x}_{n_k})_{k\in \N}$ such that
$\lim_{k\to+\infty}\tilde{x}_{n_k}=\ol x.$ But then, by using \eqref{rege1} we obtain at once that
$z_{n_k}\To (\ol x,\ol x),\,k\To+\infty,$ hence by using the fact that $\omega((\tilde{x}_n)_{n\in\N})=\omega((x_n)_{n\in\N})$ we obtain
$$\{(\ol x,\ol x)\in\R^m\times\R^m: \ol x\in\omega((x_n)_{n\in\N})\}\subseteq \omega((z_n)_{n\in\N}).$$

Finally, from Lemma \ref{H2forH} (i) and since $\omega((x_n)_{n\in\N})\subseteq\crit(g)$   we have
$$\omega((z_n)_{n\in\N})\subseteq\{(\ol x,\ol x)\in\R^m\times\R^m: \ol x\in \crit(g)\}=\crit(H).$$

(ii) Follows directly by (ii) in Lemma \ref{decreasing}.

(iii) We have:
\begin{align}\nonumber\|\n H(y_n,u_n)\|^2&=\|(\n g(y_n)+y_n-u_n,u_n-y_n)\|^2=\|\n g(y_n)+y_n-u_n\|^2+\|u_n-y_n\|^2\\
\nonumber&=\left\|\frac1s(x_n-x_{n+1})+\left(\frac{\b n}{s(n+\a)}-\sqrt{2\d_n}\right)(x_n-x_{n-1})\right\|^2+2\d_n\|x_n-x_{n-1}\|^2\\
\nonumber&\le\frac{2}{s^2}\|x_{n+1}-x_n\|^2+2\left(\left(\frac{\b n}{s(n+\a)}-\sqrt{2\d_n}\right)^2+\d_n\right)\|x_{n}-x_{n-1}\|^2,
\end{align}
 for all $n\in\N,\,n\ge N.$

 Assume now that $(x_n)_{n\in\N}$ is bounded and let us prove (iv), (see also \cite{BCL-AA}). Obviously it follows that  $(z_n)_{n\in\N}$ is also bounded, hence according to Weierstrass Theorem $\omega((z_n)_{n\in\N}),$ (and also $\omega((x_n)_{n\in\N})$), is nonempty. It remains to show that $\omega((z_n)_{n\in\N})$ is closed. From (i)  we have
\begin{equation}\label{e8}
\omega((z_n)_{n\in\N})=\{(\ol x,\ol x)\in\R^m\times \R^m: \ol x\in \omega((x_n)_{n\in\N})\},
 \end{equation}
 hence it is enough to show that $\omega((x_n)_{n\in\N})$ is closed.

 Let be $(\ol x_p)_{p\in\N}\subseteq \omega((x_n)_{n\in\N})$ and assume that $\lim_{p\To+\infty}\ol x_p=x^*.$ We show that $x^*\in\omega((x_n)_{n\in\N}).$ Obviously, for every $p\in\N$ there exists a sequence of natural numbers $n_k^p\To+\infty,\,k\To+\infty$, such that
$$\lim_{k\To+\infty}x_{n_k^p}=\ol x_p.$$

Let be $\e >0$. Since $\lim_{p\To+\infty}\ol x_p=x^*,$  there exists $P(\e) \in \N$ such that for every $p\ge P(\e)$ it holds
$$\|\ol x_p-x^*\|<\frac\e2.$$
Let $p\in\N$ be fixed. Since $\lim_{k\To+\infty}x_{n_k^p}=\ol x_p,$ there exists $k(p,\e)\in\N$ such that for every $k\ge k(p,\e)$ it holds $$\|x_{n_k^p}-\ol x_p\|<\frac\e2.$$
Let be $k_p \geq k(p,\varepsilon)$ such that $n_{k_p}^p >p$. Obviously $n_{k_p}^p \To \infty$ as $p \To +\infty$ and for every $p \geq P(\e)$
$$\|x_{n_{k_p}^p}-x^*\|\le \|x_{n_{k_p}^p}-\ol x_p\|+\|\ol x_p-x^*\|<\e.$$
Hence
$\lim_{p\To+\infty}x_{n_{k_p}^p}=x^*,$
thus $x^*\in \omega((x_n)_{n\in\N}).$

(v) By using \eqref{e8} we have
$$\lim_{n\To+\infty} \dist(z_n,\omega((z_n)_{n\in\N}))=\lim_{n\To+\infty}\inf_{\ol x\in\omega((x_n)_{n\in\N})}\|z_n-(\ol x,\ol x)\|.$$
Since there exists the subsequence $(z_{n_k})_{k\in\N}$ such that
$\lim_{k\To\infty}z_{n_k}=(\ol x_0,\ol x_0)\in \omega((z_n)_{n\in\N})$ it is straightforward that
$$\lim_{n\To+\infty} \dist(z_n,\omega((z_n)_{n\in\N}))=0.$$
\end{proof}

Now we are ready to prove Theorem \ref{convergence} concerning the convergence of the sequences generated by the numerical scheme \eqref{generaldiscrete}.

\begin{proof}({\it Proof of Theorem \ref{convergence}.}) Let $(z_n)_{n\in\N}$ be the sequence defined by \eqref{zform}. Since $x^*\in \omega((x_n)_{n\in\N})$ according to Lemma \ref{regularization} (i) one has $x^*\in\crit(g)$ and  $z^*=(x^*,x^*)\in\omega((z_n)_{n\in\N}).$

It can easily be checked that the assumptions of Theorem \ref{thabstrconv} are satisfied with  the continuously Fr\'echet differentiable function $H$, the sequences $(z_n)_{n\in\N}$ and $(\tilde{x}_n)_{n\in\N}.$ Indeed, according to \eqref{H1forH} and Lemma \ref{H2forH} (ii) the conditions (H1) and (H2) from the hypotheses of Theorem \ref{thabstrconv} are satisfied. Hence, the sequence $(\tilde{x}_n)_{n\in\N}$ converges to $x^*$ as $n\To+\infty$. But then obviously the sequence $({x}_n)_{n\in\N}$ converges to $x^*$ as $n\To+\infty$.
\end{proof}

\begin{remark} Note that under the assumptions of Theorem \ref{convergence} we also have that
$$\lim_{n\To+\infty}y_n= x^*\mbox{ and }\lim_{n\To+\infty}g(x_n)=\lim_{n\To+\infty}g(y_n)=g(x^*).$$
\end{remark}

\begin{corollary}\label{sa} In the settings of problem \eqref{opt-pb}, for some starting points $x_0,x_{-1}\in\R^m,$ consider the sequence $(x_n)_{n\in\N}$ generated by Algorithm \eqref{generaldiscrete}. Assume that $g$ semi-algebraic and bounded from below. Assume further that $\omega((x_n)_{n\in\N})\neq\emptyset.$
%Then, the sequence $(x_n)_{n\in\N}$ converges to $x^*$, $(z_n)_{n\in\N}$ converges to $z^*$ and $z^*\in\crit(F).$

 Then, the sequence $(x_n)_{n\in\N}$ converges to a critical point of the objective function $g.$
\end{corollary}
\begin{proof}
  Since the class of semi-algebraic functions is closed under addition (see for example \cite{b-sab-teb}) and
$(x,y) \mapsto \frac12\|x-y\|^2$ is semi-algebraic, we obtain that the the function
$$H:\R^m\times\R^m\To\R,\,H(x,y)=g(x)+\frac12\|y-x\|^2$$ is semi-algebraic. Consequently $H$ is a KL function. In particular $H$ has the Kurdyka-{\L}ojasiewicz property at a point $z^*=(x^*,x^*),$ where $x^*\in \omega((x_n)_{n\in\N}).$ The conclusion follows from Theorem \ref{convergence}.
\end{proof}

\begin{remark}\label{r2} In order to apply Theorem \ref{convergence} or Corollary \ref{sa} we need to assume that $\omega((x_n)_{n\in\N})$ is nonempty. Obviously, this condition is satisfied whenever the sequence $(x_n)_{n\in\N}$ is bounded. Next we show that the boundedness of $(x_n)_{n\in\N}$ is guaranteed if we assume that the objective function $g$ is coercive, that is,
$\lim_{\|x\|\rightarrow+\infty} g(x)=+\infty.$
\end{remark}
\begin{proposition}\label{xbounded} In the settings of problem \eqref{opt-pb}, for some starting points $x_0,x_{-1}\in\R^m,$ consider the sequence $(x_n)_{n\in\N}$ generated by Algorithm \eqref{generaldiscrete}. Assume that the objective function $g$ is coercive.

Then, $g$ is bounded from below, and the sequence  $(x_n)_{n\in\N}$ is bounded.
\end{proposition}
\begin{proof}
Indeed, $g$ is bounded from below, being a continuous and coercive function (see \cite{rock-wets}). Note that according to \eqref{forcoercive} the sequence $D\sum_{n=N}^r\|x_{n+1}-x_{n}\|^2$ is  bounded. Consequently,  from \eqref{forcoercive} it follows that $y_{r+1}$ is contained in a lower level set of $g$, for every $r\ge N,$ ($N$ was defined in the hypothesis of Lemma \ref{decreasing}). But the lower level sets of $g$ are bounded since $g$ is coercive. Hence, $(y_n)_{n\in\N}$ is bounded and taking into account \eqref{e7}, it follows that $(x_n)_{n\in\N}$  is also bounded.
\end{proof}

%\begin{remark} Note that, according to Remark \ref{r2}, the conclusion of Theorem \ref{convergence} remains valid if we replace in its hypotheses the conditions that $g$ is bounded from below and $\omega((z_n)_{n\in\N})$ is nonempty by the condition that $g$ is coercive. Moreover, according to Remark \ref{r2}  the set $\omega((z_n)_{n\in\N})$ is nonempty whenever the sequence $({x}_n)_{n\in\N}$ is bounded. Hence, the conclusion of Theorem \ref{convergence} remains also valid if in its hypotheses we replace the assumptions that $\omega((z_n)_{n\in\N})$ is nonempty  and $H$ has the Kurdyka-{\L}ojasiewicz property at a $z^*=(x^*,x^*)\in \omega((z_n)_{n\in\N})$  by the conditions that the sequence $({x}_n)_{n\in\N}$ is bounded and $H$ has the Kurdyka-{\L}ojasiewicz property at a point $z^*=(x^*,x^*),$ where $x^*\in \omega((x_n)_{n\in\N}).$
%\end{remark}

An immediate consequence of Theorem \ref{convergence} and Proposition \ref{xbounded} is the following result.
\begin{corollary} Assume that $g$ is a coercive function. In the settings of problem \eqref{opt-pb}, for some starting points $x_0,x_{-1}\in\R^m,$ consider the sequence $(x_n)_{n\in\N}$ generated by Algorithm \eqref{generaldiscrete}.  Assume further that
$$H:\R^m\times\R^m\To\R,\,H(x,y)=g(x)+\frac12\|y-x\|^2$$ is a KL function.

 Then, the sequence $(x_n)_{n\in\N}$ converges to a critical point of the objective function $g.$
\end{corollary}

%\section{Convergence rates}\label{sec5}

\section{Convergence rates via the {\L}ojasiewicz exponent}

In this section we will assume that the regularization function $H$, introduced in the previous section, satisfies the {\L}ojasiewicz property, which  corresponds to a particular choice of the desingularizing function $\varphi$ (see \cite{lojasiewicz1963, b-d-l2006, attouch-bolte2009,att-b-red-soub2010,b-d-l-m2010,Rchil}).
\begin{comment}
\begin{definition}\label{Lexpo}  Let $g:\R^m\To \R $ be a differentiable function. The function $g$ is said to fulfill the {\L}ojasiewicz property at $\ol x\in\crit(g)$ with  the {\L}ojasiewicz constant $K$ and {\L}ojasiewicz exponent $\t\in[0,1)$, if there exist $\e>0$  such that for every $x\in B(\|\ol x\|,\e),\mbox{ with }g(x)\neq g(\ol x),$ the KL inequality \eqref{KLineq} holds with the desingularizing function  $\varphi(t)=\frac{K}{1-\t} t^{1-\t}.$
\end{definition}
%For $\t =0$ we adopt the convention $0^0=0$, such that if $|g(x)-g(\ol{x})|^0=0$ then $g(x)=g(\ol{x}),$ (see \cite{attouch-bolte2009}).
\end{comment}

\begin{definition}\label{Lexpo}  Let $g:\R^m\To \R $ be a differentiable function. The function $g$ is said to fulfill the {\L}ojasiewicz property at a point $\ol x\in\crit(g)$ if there exist $K,\e>0$ and $\t\in[0,1)$ such that
$$|g(x)-g(\ol x)|^\t\le K\|\n g(x)\|\mbox{ for every }x\mbox{ fulfilling }\|x-\ol x\|<\e.$$
The number $K$ is called the {\L}ojasiewicz constant, meanwhile the number $\t$  is called the {\L}ojasiewicz exponent of $g$ at the critical  point $\ol x.$
\end{definition}
Note that the above definition corresponds to the case when in the KL property the desingularizing function $\varphi$ has the form $\varphi(t)=\frac{K}{1-\t} t^{1-\t}.$ For $\t =0$ we adopt the convention $0^0=0$, such that if $|g(x)-g(\ol{x})|^0=0$ then $g(x)=g(\ol{x}),$ (see \cite{attouch-bolte2009}).

\begin{comment}
 The following lemma was established in \cite{BN} and will be crucial in obtaining our convergence rates, (see also  \cite{attouch-bolte2009} for different techniques).

\begin{lemma}[\cite{BN} Lemma 15]\label{ratel} Let $(e_n)_{n\in\N}$ be  a  monotonically  decreasing  positive sequence  converging to $0.$  Assume  further that there exist the natural numbers $l_0\ge 1$ and $n_0\ge l_0$  such that for every $n\ge n_0$ one has
\begin{equation}\label{eseq}
e_{n-l_0}-e_n\ge C_0 e_n^{2\t}
\end{equation}
where $C_0>0$ is some constant and $\t\in[0,1).$  Then following  statements are true:
\begin{itemize}
\item[(i)]  if $\t=0,$  then $(e_n)_{n\ge \ol n}$ converges in finite time;
\item[(ii)]  if $\t\in\left(0,\frac12\right]$, then there exists $C_1>0$ and $Q\in[0,1)$, such that for every $n\ge n_0$
$$e_n\le C_1 Q^n;$$
\item[(iii)]  if $\t\in\left[\frac12,1\right)$, then there exists $C_2>0$, such that for every $n\ge n_0+l_0$
$$e_n\le C_2(n-l_0+1)^{-\frac{1}{2\t-1}}.$$
\end{itemize}
\end{lemma}
\end{comment}
In the following theorem we provide convergence rates for the sequences generated by \eqref{generaldiscrete}, but also for the objective function values in these sequences, in terms of the {\L}ojasiewicz exponent of the regularization $H$ (see also \cite{attouch-bolte2009,att-b-red-soub2010,ADR,b-d-l2006,GRV}).

More precisely we obtain finite convergence rates if the {\L}ojasiewicz exponent of $H$ is $0$, linear convergence rates if the {\L}ojasiewicz exponent of $H$ belongs to $\left(0,\frac12\right]$ and sublinear convergence rates if the {\L}ojasiewicz exponent of $H$ belongs to $\left(\frac12,1\right).$

\begin{theorem}\label{th-conv-rate}  In the settings of problem \eqref{opt-pb}, for some starting points $x_0,x_{-1}\in\R^m,$ consider the sequences $(x_n)_{n\in\N},\,(y_n)_{n\in\N}$ generated by Algorithm \eqref{generaldiscrete}. Assume that $g$ is bounded from below and consider the function
$$H:\R^m \times \R^m \To \R,\,H(x,y)=g(x)+\frac12\|x-y\|^2.$$
Let $(z_n)_{n\in\N}$ be the sequence defined by \eqref{zform} and assume that $\omega((z_n)_{n\in\N})$ is nonempty and that $H$ fulfills the {\L}ojasiewicz property  with  {\L}ojasiewicz constant $K$ and   {\L}ojasiewicz  exponent $\t\in\left[0,1\right)$  at a point $z^*=(x^*,x^*)\in \omega((z_n)_{n\in\N}).$
Then $\lim_{n\To+\infty}x_n=x^*\in\crit(g)$ and the following statements hold true:
\\
If $\t=0$ then
\begin{itemize}
\item[$(\emph{a}_0)$]  $(g(y_n))_{n\in\N},(g(x_n))_{n\in\N},(y_n)_{n\in\N}$ and $(x_n)_{n\in\N}$ converge in a finite number of steps;
\end{itemize}
If $\t\in\left(0,\frac12\right]$ then there exist $Q\in[0,1),$  $a_1,a_2,a_3,a_4>0$ and $\ol k\in\N$ such that
\begin{itemize}
\item[$(\emph{a}_1)$] $g(y_{n})-g(x^* )\le a_1 Q^n$ for  every $n\ge \ol k$,
\item[$(\emph{a}_2)$] $g(x_{n})-g(x^* )\le a_2 Q^n$ for  every $n\ge\ol k$,
\item[$(\emph{a}_3)$] $\|x_n-x^* \|\le a_3 Q^{\frac{n}{2}}$ for  every $n\ge\ol k$,
\item[$(\emph{a}_4)$] $\|y_n-x^* \|\le a_4 Q^{\frac{n}{2}}$ for all  $n\ge\ol k$;
\end{itemize}
If $\t\in\left(\frac12,1\right)$ then there exist  $b_1,b_2,b_3,b_4>0$ and $\ol k\in\N$ such that
\begin{itemize}
\item[$(\emph{b}_1)$] $g(y_n)-g(x^* )\le b_1 {n^{-\frac{1}{2\t-1}}},\mbox{ for all }n\ge\ol k$,
\item[$(\emph{b}_2)$] $g(x_n)-g(x^* )\le b_2 {n^{-\frac{1}{2\t-1}}},\mbox{ for all }n\ge\ol k$,
\item[$(\emph{b}_3)$] $\|x_{n}-x^* \|\le b_3 {n^{\frac{\t-1}{2\t-1}}},\mbox{ for all }n\ge\ol k$,
\item[$(\emph{b}_4)$] $\|y_n-x^* \|\le  b_4 {n^{\frac{\t-1}{2\t-1}}},\mbox{ for all }n\ge\ol k$.
\end{itemize}
\end{theorem}
Due to the technical details of the proof of Theorem \ref{th-conv-rate}, we will first present a sketch
of it in order to give a better insight.

 1. After discussing a straightforward case, we introduce the discrete energy  $\mathcal{E}_n=H(z_n)-H(z^*)$  where $\mathcal{E}_n>0$ for all $n\in\N$, and we show that  Lemma 15 from \cite{BN} (see also  \cite{attouch-bolte2009}), %Lemma \ref{ratel}
 can be applied to $\mathcal{E}_n.$

 2. This immediately gives the desired convergence rates (a$_0$), (a$_1$) and (b$_1$).

 3. For proving  (a$_2$) and (b$_2$) we use the identity $g(x_n)-g(x^* )=(g(x_n)-g(y_n))+(g(y_n)-g(x^* ))$ and we derive an inequality between
 $g(x_n)-g(y_n))$ and $\mathcal{E}_n.$

 4. For (a$_3$) and (b$_3$) we use the equation \eqref{fontos} and the form of the desingularizing function $\varphi$.

 5. Finally, for proving (a$_4$) and (b$_4$) we use the results  already obtained at (a$_3$) and (b$_3$) and the form of the sequence $(y_n)_{n\in\N}.$

We now pass to a detailed presentation of this proof.
\begin{proof}
Obviously, according to Theorem \ref{convergence} one has $\lim_{n\To+\infty}x_n=x^*\in\crit(g)$, which combined with Lemma \ref{regularization} (i) furnishes  $\lim_{n\To+\infty}z_n=z^*\in\crit(H).$ We divide the proof into two cases.

{\bf Case I.} Assume that there exists $\ol n\in\N,$  such that $H(z_{\ol n})=H(z^*).$ Then, since $H(z_n)$ is decreasing for all $n\in\N$ and $\lim_{n\To+\infty}H(z_n)=H(z^*)$ we obtain that
$$H(z_n)=H(z^*)\mbox{ for all }n\ge\ol n.$$

The latter relation combined with \eqref{H1forH} leads to
$$0\le D\|\tilde{x}_{n+1}-\tilde{x}_{n}\|^2\le H(z_n)-H(z_{n+1})=H(z^*)-H(z^*)=0$$
for all $n\ge\ol n.$

Hence $(\tilde{x}_n)_{n\ge\ol n}$ is constant, in other words $x_n=x^*$ for all ${n\ge\ol n+N}$. Consequently $y_n=x^*$ for all ${n\ge\ol n+N+1}$ and the conclusion of the theorem is straightforward.

{\bf Case II.} In what follows we assume  that $H(z_{n})>H(z^*),$ for all $n\in\N.$

For simplicity let us denote $\mathcal{E}_n=H(z_n)-H(z^*)$  and observe that $\mathcal{E}_n>0$ for all $n\in\N.$
 From \eqref{e61} we have that the sequence $(\mathcal{E}_n)_{n\in\N}$ is nonincreasing, that is, there exists $D>0$ such  that
 \begin{equation}\label{nonincr}
D\|\tilde{x}_n-\tilde{x}_{n-1}\|^2\le \mathcal{E}_n-\mathcal{E}_{n+1},\mbox{ for all }n\in\N.
\end{equation}

Further, since $\lim_{n\To+\infty}z_n=z^*, $ one has
\begin{equation}\label{conv0}
\lim_{n\To+\infty}\mathcal{E}_n=\lim_{n\To+\infty}(H(z_n)-H(z^* ))=0.
\end{equation}

From Lemma \ref{regularization} (iii) we have
 \begin{equation}\label{gradh}
 \|\n H(z_n)\|^2\le \frac{2}{s^2}\|\tilde{x}_{n+1}-\tilde{x}_n\|^2+2\left(\left(\frac{\b (n+N)}{s(n+N+\a)}-\sqrt{2\d_{n+N}}\right)^2+\d_{n+N}\right)\|\tilde{x}_{n}-\tilde{x}_{n-1}\|^2,
 \end{equation} for all $n\in\N$.
  Let $S_n=2\left(\frac{\b (n+N)}{s(n+N+\a)}-\sqrt{2\d_{n+N}}\right)^2+\d_{n+N},$ for all $n\in\N.$

Combining \eqref{nonincr} and \eqref{gradh} it follows that, for all $n\in\N$ one has
\begin{align}\label{inter0}
\|\tilde{x}_n-\tilde{x}_{n-1}\|^2&\ge \frac{1}{S_n}\|\n H(z_n)\|^2-\frac{2}{s^2S_n}\|\tilde{x}_{n+1}-\tilde{x}_n\|^2\\
\nonumber&\ge \frac{1}{S_n}\|\n H(z_n)\|^2-\frac{2}{s^2 DS_n}(\mathcal{E}_{n+1}-\mathcal{E}_{n+2}).
\end{align}

Now by using the {\L}ojasiewicz property of $H$ at $z^*\in\crit(H)$, and the fact that $\lim_{n\To+\infty}z_n=z^*$,  we obtain that there exists  $\e>0$ and $\ol N_1\in\N,$ such that for all $n\ge \ol N_1$ one has
$$\|z_n-z^*\|<\e,$$
and
\begin{equation}\label{inter2}
\|\n H(z_n)\|^2\ge\frac{1}{K^2}|H(z_n)-H(z^* )|^{2\t}=\frac{1}{K^2}\mathcal{E}_n^{2\t}.
\end{equation}
Consequently \eqref{nonincr}, \eqref{inter0} and \eqref{inter2} leads to
\begin{align}\label{e111} \mathcal{E}_n-\mathcal{E}_{n+1}&\ge D\|\tilde{x}_n-\tilde{x}_{n-1}\|^2\ge \frac{D}{S_n}\|\n H(z_n)\|^2-\frac{2}{s^2S_n}(\mathcal{E}_{n+1}-\mathcal{E}_{n+2})\\
\nonumber&\ge \frac{D}{K^2 S_n}\mathcal{E}_n^{2\t}-\frac{2}{s^2S_n}(\mathcal{E}_{n+1}-\mathcal{E}_{n+2})=a_n \mathcal{E}_{n}^{2\t}-b_n(\mathcal{E}_{n+1}-\mathcal{E}_{n+2}),
\end{align}
for all $n\ge \ol N_1,$ where $a_n= \frac{D}{K^2 S_n}\mbox{ and  }b_n=\frac{2}{s^2S_n}.$

Since the sequence $(\mathcal{E}_n)_{n\in\N}$ is nonincreasing, one has
$$\mathcal{E}_n-\mathcal{E}_{n+2}\ge\mathcal{E}_n-\mathcal{E}_{n+1},$$
$$a_n \mathcal{E}_{n}^{2\t}\ge a_n \mathcal{E}_{n+2}^{2\t}$$
and
$$-b_n(\mathcal{E}_{n+1}-\mathcal{E}_{n+2})\ge-b_n(\mathcal{E}_{n}-\mathcal{E}_{n+2}),$$
thus, \eqref{e111} becomes
\begin{equation}\label{e1112}
\mathcal{E}_n-\mathcal{E}_{n+2}\ge \frac{a_n}{1+b_n}\mathcal{E}_{n+2}^{2\t},
\end{equation}
for all $n\ge \ol{N_1}.$

It is obvious that the sequences $(a_n)_{n\ge\ol N_1}$ and $(b_n)_{n\ge\ol N_1}$ are positive and convergent, further
$$\lim_{n\To+\infty} a_n>0\mbox{ and } \lim_{n\To+\infty} b_n>0,$$
 hence, there exists $\ol N_2\in\N,\, \ol N_2\ge \ol N_1,$ and $C_0>0$ such that
 $$\frac{a_n}{1+b_n}\ge C_0,\mbox{ for all }n\ge \ol N_2.$$

Consequently, \eqref{e1112} leads to
\begin{equation}\label{e1113}
\mathcal{E}_n-\mathcal{E}_{n+2}\ge C_0\mathcal{E}_{n+2}^{2\t},
\end{equation}
for all $n\ge \ol{N_2}.$

Now we can apply Lemma 15 \cite{BN} %Lemma \ref{ratel} by observing that \eqref{e1113} is nothing else that \eqref{eseq} in Lemma \ref{ratel},
with $e_n=\mathcal{E}_{n+2},\,l_0=2$ and $n_0=\ol N_2.$
Hence, by taking into account that $\mathcal{E}_n>0$ for all $n\in\N$, that is, in the conclusion of Lemma 15 (ii) from \cite{BN} %Lemma \ref{ratel} (ii)
one has $Q\neq 0$,   we have:
\begin{itemize}
\item[(K0)]if $\t=0,$  then $(\mathcal{E}_n)_{n\ge N}$ converges in finite time;
\item[(K1)]  if $\t\in\left(0,\frac12\right]$, then there exists $C_1>0$ and $Q\in(0,1)$, such that for every $n\ge \ol N_2+2$
$$\mathcal{E}_n\le C_1 Q^n;$$
\item[(K2)]  if $\t\in\left[\frac12,1\right)$, then there exists $C_2>0$, such that for every $n\ge \ol N_2+4$
$$\mathcal{E}_n\le C_2(n-3)^{-\frac{1}{2\t-1}}.$$
\end{itemize}

{\bf The case $\t=0.$}

For proving (a$_0$) we use (K0). Since in this case $(\mathcal{E}_n)_{n\ge N}$ converges in finite time after an index $N_0\in\N$ we have $\mathcal{E}_n-\mathcal{E}_{n+1}=0$ for all $n\ge N_0$, hence \eqref{nonincr} implies that $\tilde{x}_n=\tilde{x}_{n-1}$ for all $n\ge N_0$. Consequently, $x_n=x_{n-1}$ and $y_n=x_n$ for all $n\ge N_0+N$, thus $(x_n)_{n\in\N},\, (y_n)_{n\in\N}$ converge in finite time which obviously implies that
$\left(g(x_n)\right)_{n\in\N},\,\left(g(y_n)\right)_{n\in\N}$ converge in finite time.
\vskip 0.3cm
{\bf The case $\t\in\left(0,\frac12\right].$}

We apply (K1) and we obtain that there exists $C_1>0$ and $Q\in(0,1)$, such that for every $n\ge \ol N_2+2$, one has $\mathcal{E}_n\le C_1 Q^n.$
But, $\mathcal{E}_n=g(\tilde{y}_n) - g(x^*) + \delta_{n+N} \| \tilde{x}_n - \tilde{x}_{n-1} \|^2$ for all $n\in\N,$ consequently
$g(y_{n+N}) - g(x^*)\le  C_1 Q^n,$ for all $n\ge\ol N_2+2.$
Thus, by denoting $\frac{C_1}{Q^N}=a_1$ we get
\begin{align}\label{ea1}
g(y_n) - g(x^*)\le  a_1 Q^n,\mbox{ for all }n\ge\ol N_2+N+2.
\end{align}
For (a$_2$) we start from  \eqref{desc} and Algorithm \eqref{generaldiscrete} and for all $n\in\N$ we have
\begin{align}\nonumber g(x_n)-g(y_n)&\le\<\n g(y_n),x_n-y_n\>+\frac{L_g}{2}\|x_n-y_n\|^2\\
\nonumber&=\frac1s\left\<(x_n-x_{n+1})+\frac{\b n}{n+\a}(x_n-x_{n-1}),-\frac{\b n}{n+\a}(x_n-x_{n-1})\right\>\\
\nonumber&+\frac{L_g}{2}\left(\frac{\b n}{n+\a}\right)^2\|x_n-x_{n-1}\|^2\\
\nonumber&=-\left(\frac{\b n}{n+\a}\right)^2\frac{2-sL_g}{2s}\|x_n-x_{n-1}\|^2+\frac1s\left\<x_{n+1}-x_n,\frac{\b n}{n+\a}(x_n-x_{n-1})\right\>.
\end{align}
By using the inequality $\<X,Y\>\le \frac12\left(a^2\|X\|^2+\frac{1}{a^2}\|Y\|^2\right)$ for all $X,Y\in\R^m, a\in \R\setminus\{0\}$, we obtain
$$\left\<x_{n+1}-x_n,\frac{\b n}{n+\a}(x_n-x_{n-1})\right\>\le \frac12\left(\frac{1}{2-sL_g}\|x_{n+1}-x_n\|^2+(2-sL_g)\left(\frac{\b n}{n+\a}\right)^2\|x_n-x_{n-1}\|^2\right),$$ consequently
\begin{equation}\label{imagedifference}
g(x_n)-g(y_n)\le\frac{1}{2s(2-sL_g)}\|x_{n+1}-x_n\|^2,\mbox{ for all }n\in\N.
\end{equation}

Taking into account that $\mathcal{E}_n>0$ for all $n\in\N,$ from \eqref{nonincr} we have
\begin{equation}\label{ee}
\|\tilde{x}_{n}-\tilde{x}_{n-1}\|^2\le\frac{1}{{D}}{\mathcal{E}_{n}}\mbox{ for all }n\in\N.
\end{equation}

Hence, for all $n\ge N-1$ one has
\begin{equation}\label{ee1}
g(x_n)-g(y_n)\le\frac{1}{2sD(2-sL_g)}\mathcal{E}_{n-N+1}.
\end{equation}

Now, the identity $g(x_n)-g(x^* )=(g(x_n)-g(y_n))+(g(y_n)-g(x^* ))$ and  \eqref{ea1} lead to
$$g(x_n)-g(x^* )\le \frac{1}{2sD(2-sL_g)}\mathcal{E}_{n-N+1}+a_1 Q^n$$
for every $n\ge\ol N_2+N+2$, which combined with (K1) gives
\begin{equation}\label{ea2}
g(x_n)-g(x^* )\le a_2 Q^n,
\end{equation}
for every $n\ge\ol N_2+N+2$, where $a_2=\frac{C_1}{2sD(2-sL_g)Q^{N-1}}+a_1.$

For (a$_3$) we will use \eqref{fontos}. Since $z_n\in B(z^*,\e)$ for all $n\ge \ol N_2$ and $z_n\To z^*,\, n\To+\infty$, we get that there exists $\ol N_3\in\N,\,\ol N_3\ge\ol N_2$ such that \eqref{fontos} holds for every $n\ge \ol N_3.$  In this setting, by taking into account that the desingularizing function is $\varphi(t)=\frac{K}{1-\t}t^{1-\t},$ the inequality \eqref{fontos} has the form
\begin{equation}\label{fontoska}
2\|\tilde{x}_{k+1}-\tilde{x}_{k}\|\le\|\tilde{x}_{k}-\tilde{x}_{k-1}\|+\frac{9b}{4 D}\cdot\frac{K}{1-\t}(\mathcal{E}_k^{1-\t}-\mathcal{E}_{k+1}^{1-\t}),
\end{equation}
where $D$ was defined at \eqref{H1forH} and $b$ was defined at Lemma \ref{H2forH} (ii).
Observe that  by summing up \eqref{fontoska} from $k=n\ge\ol N_3$ to $k=P>n$ and using the triangle inequality we obtain
\begin{align}
\nonumber\|\tilde{x}_{P+1}-\tilde{x}_{n}\|&\le \sum_{k={n}}^P \|\tilde{x}_{k+1}-\tilde{x}_{k}\|\\
\nonumber&\le \|\tilde{x}_{n}-\tilde{x}_{n-1}\|-\|\tilde{x}_{P+1}-\tilde{x}_P\|+\frac{9b}{4D}\cdot\frac{K}{1-\t}(\mathcal{E}_n^{1-\t}-\mathcal{E}_{P+1}^{1-\t}).
\end{align}

By  letting $P\To+\infty$ and taking into account that $\tilde{x}_P\To x^*,\, \mathcal{E}_{P+1}\To 0,\,P\To+\infty,$ further  using \eqref{ee} we get
\begin{equation}\label{fontoska1}
\|\tilde{x}_{n}-x^* \|\le \|\tilde{x}_{n}-\tilde{x}_{n-1}\|+\frac{9b}{4D}\cdot\frac{K}{1-\t}\mathcal{E}_n^{1-\t}\le\frac{1}{\sqrt{D}}\sqrt{\mathcal{E}_n}+ M_0\mathcal{E}_n^{1-\t},
\end{equation}
where $M_0=\frac{9bK}{4D(1-\t)}.$

But $(\mathcal{E}_n)_{n\in\N}$ is a decreasing sequence and according to \eqref{conv0} $(\mathcal{E}_n)_{n\in\N}$ converges to $0$, hence there exists $\ol N_4\ge \max\{\ol N_3,\ol N_2+2\}$ such that $0\le \mathcal{E}_{n}\le 1,$ for all $n\ge \ol N_4$. The latter relation combined with the fact that $\t\in\left(0,\frac12\right]$ leads to
$\mathcal{E}_{n}^{1-\t}\le \sqrt{\mathcal{E}_{n}},\mbox{ for all }n\ge \ol N_4.$ Consequently we have
$\|\tilde{x}_{n}-x^* \|\le M_1\sqrt{\mathcal{E}_{n}},$ for all $n\ge\ol N_4,$ where $M_1=\frac{1}{\sqrt{D}}+M_0.$
The conclusion follows via (K1) since  we have
$$\|x_{n+N}-x^* \|\le M_1\sqrt{C_1} Q^{\frac{n}{2}}=M_1\sqrt{\frac{C_1}{Q^{N}}}Q^{\frac{n+N}{2}},\mbox{ for every }n\ge\ol N_4$$
and consequently
\begin{equation}\label{ea3}
\|x_{n}-x^* \|\le a_3 Q^{\frac{n}{2}},
\end{equation}
for all $n\ge \ol N_4+N,$ where $a_3=M_1\sqrt{\frac{C_1}{Q^{N}}}.$

Finally, for $n\ge \ol N_4+N+1$ we have
\begin{align}\nonumber \|y_n-x^* \|&=\left\|x_n+\frac{\b n}{n+\a}(x_n-x_{n-1})-x^* \right\|\le\left(1+\frac{\b  n}{n+\a}\right)\|x_n-x^* \|+\frac{\b  n}{n+\a}\|x_{n-1}-x^* \|\\
\nonumber&\le\left(1+\frac{\b  n}{n+\a}\right)a_3Q^{\frac{n}{2}}+\frac{\b  n}{n+\a}a_3Q^{\frac{n-1}{2}}=\left(1+\frac{\b  n}{n+\a}+\frac{\b  n}{n+\a}\frac{1}{\sqrt{Q}}\right)a_3Q^{\frac{n}{2}}.
\end{align}
Let $a_4=\left(1+\b+\frac{\b}{\sqrt{Q}}\right)a_3.$ Then, for all $n\ge \ol N_4+N+1$ one has
\begin{equation}\label{ea4}
\|y_n-x^* \|\le a_4Q^{\frac{n}{2}}.
\end{equation}
Now, if we take  $\ol k= \ol N_4+N+1$ then \eqref{ea1}, \eqref{ea2}, \eqref{ea3} and \eqref{ea4} lead to  the conclusions (a$_1)$-(a$_4)$.
\vskip 0.3cm
{\bf The case $\t\in\left(\frac12,1\right).$}

 According to (K2) there exists $C_2>0$, such that for every $n\ge \ol N_2+4$ one has
$$\mathcal{E}_n\le C_2(n-3)^{-\frac{1}{2\t-1}}=C_2\left(\frac{n}{n-3}\right)^{\frac{1}{2\t-1}}n^{-\frac{1}{2\t-1}}.$$
Let $M_2=C_2\sup_{n\ge \ol N_2+4}\left(\frac{n}{n-3}\right)^{\frac{1}{2\t-1}}=C_2\left(\frac{\ol N_2+4}{\ol N_2+1}\right)^{\frac{1}{2\t-1}}.$
Then,
$\mathcal{E}_n\le M_2 n^{-\frac{1}{2\t-1}},$ for all $n\ge \ol N_2+4.$
But $\mathcal{E}_n=g(\tilde{y}_n) - g(x^*) + \delta_{n+N} \| \tilde{x}_n - \tilde{x}_{n-1} \|^2$,
hence $g(\tilde{y}_n) - g(x^*)\le M_2 n^{-\frac{1}{2\t-1}},$ for every $n\ge \ol N_2+4.$
Consequently, for every $n\ge \ol N_2+4$  we have
$$g({y}_{n+N}) - g(x^*)\le M_2\left(\frac{n+N}{n}\right)^{\frac{1}{2\t-1}}( n+N)^{-\frac{1}{2\t-1}}.$$
Let $b_1=M_2\left(\frac{N_2+4+N}{N_2+4}\right)^{\frac{1}{2\t-1}}=C_2\left(\frac{\ol N_2+4}{\ol N_2+1}\right)^{\frac{1}{2\t-1}}\left(\frac{N_2+4+N}{N_2+4}\right)^{\frac{1}{2\t-1}}=C_2\left(\frac{N_2+4+N}{\ol N_2+1}\right)^{\frac{1}{2\t-1}}.$
Then,
\begin{equation}\label{eb1}
g({y}_{n}) - g(x^*)\le b_1 n^{-\frac{1}{2\t-1}},\mbox{ for all }n\ge N_2+N+4.
\end{equation}

For (b$_2$) note that \eqref{ee1} holds for every  $n\ge \ol N_2+4,$ hence
$$
g(x_n)-g(y_n)\le\frac{1}{2sD(2-sL_g)}\mathcal{E}_{n-N+1}.
$$
Further,
$$\mathcal{E}_{n-N+1}\le M_2 (n-N+1)^{-\frac{1}{2\t-1}}=M_2 \left(\frac{n}{n-N+1}\right)^{\frac{1}{2\t-1}}n^{-\frac{1}{2\t-1}}\le C_2\left(\frac{\ol N_2+N+3}{\ol N_2+1}\right)^{\frac{1}{2\t-1}}n^{-\frac{1}{2\t-1}},$$ for all $n\ge \ol N_2+N+3.$
Consequently, $g(x_n)-g(y_n)\le M_3n^{-\frac{1}{2\t-1}},$ for all $n\ge \ol N_2+N+3,$ where
$M_3=\frac{1}{2sD(2-sL_g)}C_2\left(\frac{\ol N_2+N+3}{\ol N_2+1}\right)^{\frac{1}{2\t-1}}.$
Therefore, by using the latter inequality and \eqref{eb1} one has
$$
g(x_n) - g(x^*) = (g(x_n) - g(y_n)) + (g(y_n)-g(x^*) \leq \left(M_3+b_1\right)n^{\frac{-1}{2\t-1}},
\mbox{ for every  }n\ge \ol N_2+N+4.$$
Let $b_2=M_3+b_1.$
Then,
\begin{equation}\label{eb2}
g(x_n) - g(x^*)\le b_2 n^{\frac{-1}{2\t-1}},\mbox{ for every }n\ge \ol N_2+N+4.
\end{equation}

For   (b$_3$) we use \eqref{fontoska1} again. Arguing as at (a$_3$) we obtain that $0\le \mathcal{E}_{n}\le 1,$ for all $n\ge \ol N_5,$ where $\ol N_5=  \max\{\ol N_4,\ol N_2+4\}$.

Now, by using  the fact that $\t\in\left(\frac12,1\right),$ we get that
$\mathcal{E}_{n}^{1-\t}\ge \sqrt{\mathcal{E}_{n}},\mbox{ for all }n\ge \ol N_5.$  Consequently, from \eqref{fontoska1} we get
$\|\tilde{x}_{n}-x^* \|\le M_1{\mathcal{E}_{n}^{1-\t}},$ for all $n\ge\ol N_5.$ 

Since
$  \mathcal{E}_{n}^{1-\t}\le (M_2 n^{\frac{-1}{2\t-1}})^{1-\t},$ for all $n\ge \ol N_5$
we get
$$ \| {x}_{n+N} -x^* \|\le M_1 M_2^{\frac{\t-1}{2\t-1}}\left(\frac{n}{n+N}\right)^{\frac{\t-1}{2\t-1}}(n+N)^\frac{\t-1}{2\t-1},\mbox{ for all }n\ge \ol N_5.$$
Let $b_3=M_1 M_2^{\frac{\t-1}{2\t-1}}\left(\frac{\ol N_5}{\ol N_5+N}\right)^{\frac{\t-1}{2\t-1}}.$
Then,
\begin{equation}\label{eb3}
\| x_n - x^* \|\le b_3 n^{\frac{\t-1}{2\t-1}},\mbox{ for all }n\ge \ol N_5+N.
\end{equation}

For the  final estimate observe that for all $n\ge\ol N_5+N+1$ one has
\begin{align}\nonumber\| y_n - x^* \| &= \left\| x_n + \frac{\b n}{n+\a}(x_n - x_{n-1}) - x^* \right\| \leq \left(1+\frac{\b n}{n+\a}\right) \cdot \| x_n - x^* \| + \frac{\b n}{n+\a} \cdot \| x_{n-1} - x^* \|\\
\nonumber&\le\left(1+  \frac{\b n}{n+\a} \right) b_3 n^{\frac{\t-1}{2\t-1}} +  \frac{\b n}{n+\a}  b_3 (n-1)^{\frac{\t-1}{2\t-1}} \leq  \left(1+ 2 \frac{\b n}{n+\a} \right) b_3 (n-1)^{\frac{\t-1}{2\t-1}}.
\end{align}
Let $b_4 =b_3\sup_{n\ge \ol N_5+N+1}\left(1+ 2\frac{\b n}{n+\a} \right)(\frac{n}{n-1})^{\frac{1-\t}{2\t-1}}> 0$. Then,
\begin{align}\label{eb4}
\|  y_n - x^* \| \leq b_4 n^{\frac{\t-1}{2\t-1}}, \mbox{ for all }n\ge\ol N_5+N+1.
\end{align}
Now, if we take $\ol k=\ol N_5+N+1$ then \eqref{eb1}, \eqref{eb2}, \eqref{eb3} and \eqref{eb4} lead to  the conclusions (b$_1)$-(b$_4)$.
\end{proof}

\begin{remark}\label{r3} According to \cite{LP} (see also \cite{BBJ}), there are situations when it is enough to assume that the objective function $g$ has the {\L}ojasiewicz property instead of considering this assumption for the regularization function $H.$  More precisely in \cite{LP} it was obtained the following result, reformulated to our setting.
\end{remark}

\begin{proposition}[Theorem 3.6. \cite{LP}]\label{pLP} Suppose that $g$ has the {\L}ojasiewicz property with {\L}ojasiewicz exponent $\t\in\left[\frac12,1\right)$ at $\ol x\in\R^m.$ Then the function $H:\R^m\times\R^m\To\R,\,H(x,y)=g(x)+\frac12\|y-x\|^2$ has the {\L}ojasiewicz property at $(\ol x,\ol x)\in\R^m\times\R^m$ with the same {\L}ojasiewicz exponent $\t.$
\end{proposition}

\begin{corollary} In the settings of problem \eqref{opt-pb}, for some starting points $x_0,x_{-1}\in\R^m,$  consider the sequences $(x_n)_{n\in\N},\,(y_n)_{n\in\N}$ generated by Algorithm \eqref{generaldiscrete}. Assume that $g$ is bounded from below and $g$ has the {\L}ojasiewicz property at $x^*\in\omega((x_n)_{n\in\N})$, (which obviously must be assumed  nonempty), with {\L}ojasiewicz exponent $\t\in\left[\frac12,1\right)$. If $\t=\frac12$ then the convergence rates (a$_1)$-(a$_4)$,  if  $\t\in\left(\frac12,1\right)$  then the convergence rates (b$_1)$-(b$_4)$ stated in the conclusion of Theorem \ref{th-conv-rate} hold.
\end{corollary}
\begin{proof}
 Indeed, from Lemma \ref{regularization} (i) one has $z^*=(x^*,x^*)\in\omega((z_n)_{n\in\N})$ and according to Proposition \ref{pLP}
$H$ has the {\L}ojasiewicz property at $z^*$ with {\L}ojasiewicz exponent $\t.$ Hence, Theorem \ref{th-conv-rate} can be applied. %Notice further that the conditions that $g$ is bounded from below and $\omega((z_n)_{n\in\N})\neq\emptyset$ are fulfilled provided the objective function $g$ is coercive.
\end{proof}

As an easy consequence  of Theorem \ref{th-conv-rate} we obtain linear convergence rates for the sequences generated by Algorithm \eqref{generaldiscrete} in the case when the objective function $g$ is strongly convex.
 For similar results  concerning Polyak's algorithm and ergodic convergence rates we refer to \cite{SYLHGJ} and  \cite{GFJ}.

\begin{theorem}\label{ratestrconv}
In the settings of problem \eqref{opt-pb}, for some starting points $x_0,x_{-1}\in\R^m,$  consider the sequences $(x_n)_{n\in\N},\,(y_n)_{n\in\N}$ generated by Algorithm \eqref{generaldiscrete}. Assume that $g$ is strongly convex and let $x^*$ be the unique minimizer of $g.$ Then, there exists $Q\in[0,1)$ and there exist $a_1,a_2,a_3,a_4>0$ and $\ol k\in\N$ such that the following statements hold true:
\begin{itemize}
\item[$(\emph{a}_1)$] $g(y_{n})-g(x^* )\le a_1 {Q^n}$ for  every $n\ge \ol k$,
\item[$(\emph{a}_2)$] $g(x_{n})-g(x^* )\le a_2{Q^n}$ for  every $n\ge\ol k$,
\item[$(\emph{a}_3)$] $\|x_n-x^* \|\le a_3 {Q^{\frac{n}{2}}}$ for  every $n\ge\ol k$,
\item[$(\emph{a}_4)$] $\|y_n-x^* \|\le a_4 {Q^{\frac{n}{2}}}$ for all  $n\ge\ol k$.
\end{itemize}
\end{theorem}
\begin{proof}
We emphasize that a strongly convex function  is coercive, see \cite{BausC}. According to Proposition \ref{xbounded} the function $g$ is bounded from bellow. According to \cite{attouch-bolte2009}, $g$ satisfies the {\L}ojasiewicz property at $x^*$ with the {\L}ojasiewicz exponent $\t=\frac12.$ Then, according to Proposition \ref{pLP}, $H$ satisfies the {\L}ojasiewicz property at $(x^*,x^*)$ with the {\L}ojasiewicz exponent $\t=\frac12.$ The conclusion now follows from Theorem \ref{th-conv-rate}.
\end{proof}

\subsection{Numerical experiments}\label{ssecA2}

The aim of this section is to highlight via numerical experiments some interesting features of the generic Algorithm \eqref{generaldiscrete}.
In order to give a better perspective on the advantages and disadvantages of Algorithm \eqref{generaldiscrete} for different choices of stepsizes and inertial coefficients, in our  numerical experiments we consider the  following algorithms, all associated to the minimization problem \eqref{opt-pb}.

{(a)} A particular form of Nesterov's algorithm, (see \cite{ch-do2015}), that is,
\begin{equation}\label{nest1}
 x_{n+1}=x_n+\frac{ n}{n+3}(x_n-x_{n-1})-s\n g\left(x_n+\frac{n}{n+3}(x_n-x_{n-1})\right),
 \end{equation}
where $s=\frac{1}{L_g}.$ According to \cite{nesterov83,ch-do2015}, for optimization problems with convex objective function, Algorithm \eqref{nest1} provides convergence rates of order $\mathcal{O}\left(\frac{1}{n^2}\right)$ for the energy error $g(x_n)-\min g.$

{(b)}  Nesterov's algorithm associated to optimization problems with a $\mu-$strongly convex function (see \cite{Nest}), that is,
\begin{equation}\label{neststr1}
 x_{n+1}=x_n+\frac{ \sqrt{L_g}-\sqrt{\mu}}{\sqrt{L_g}+\sqrt{\mu}}(x_n-x_{n-1})-s\n g\left(x_n+\frac{n}{n+3}(x_n-x_{n-1})\right),
 \end{equation}
where $s=\frac{1}{L_g}.$ According to  \cite{Nest} Algorithm \eqref{neststr1} provides linear convergence rates.

{(c)} The  gradient descent algorithm with a $\mu-$strongly convex objective function, that is,
\begin{equation}\label{graddesc}
 x_{n+1}=x_n-s\n g\left(x_n\right),
 \end{equation}
 where the maximal admissible stepsize is $s=\frac{2}{L_g+\mu}.$
According to some recent results \cite{ADR} depending by the geometrical properties of the objective function the gradient descent method  may have a better convergence rate than Algorithm \eqref{nest1}.

{(d)} A particular form of Polyak's algorithm (see \cite{OCBP,BCL1}), that is,
\begin{equation}\label{polyak1}
x_{n+1}=x_n+\frac{\b n}{n+3}(x_n-x_{n-1})-s\n g\left(x_n\right),
\end{equation}
with $\b\in(0,1)$ and $0<s<\frac{2(1-\b)}{L_g}.$ For a strongly convex objective function this algorithm provides linear convergence rates \cite{Ochs}.

{(e)} Algorithm \eqref{generaldiscrete} studied in this paper, with $\a=3$, which in the view of Theorem \ref{ratestrconv} assures linear convergence rates whenever the objective function is strongly convex. %, that is
%$$ x_{n+1}=x_n+\frac{\b n}{n+3}(x_n-x_{n-1})-s\n g\left(x_n+\frac{\b n}{n+3}(x_n-x_{n-1})\right),$$
%where $\b\in(0,1)$ and $0<s<\frac{2(1-\b)}{L_g}.$
\vskip0.5cm

{\bf 1.} In our first numerical experiment we consider as an objective function the  strongly convex function $$g:\R^2\To\R,\,g(x,y)= 8x^2 + 50y^2.$$
Since $\n g(x,y)=(16x,100y)$, we infer that the Lipschitz constant of its gradient is $L_g =100$ and the strong convexity parameter $\mu=16$. Observe that the global minimum of $g$ is attained at $(0,0),$ further $g(0,0)=0.$

 Obviously, for this choice of the objective function $g$, the stepsize will become $s=\frac{1}{L_g}=0.01$ in Algorithm \eqref{nest1} and Algorithm \eqref{neststr1}, meanwhile in Algorithm \eqref{graddesc} the stepsize is $s=\frac{2}{L_g+\mu}\approx  0.0172$.

In Algorithm \eqref{polyak1} and Algorithm \eqref{generaldiscrete} we consider the instances $$(\b,s)\in\{(0.33,0.0133), (0.5,0.009), (0.66,0.0067)\}.$$ Obviously for these values we have $\b\in(0,1)$ and $0<s<\frac{2(1-\b)}{L_g}.$

We run the simulations, by considering  the same starting points $x_0=x_{-1}=(1,-1)\in\R^2$ until the energy error $|g(x_{n})-\min g|$ attains the value $10^{-150}$. The results are shown in Figure 1, where the horizontal axis measures the number of iterations and the vertical axis  shows the error $|g(x_{n})-\min g|$.

\begin{figure}[hbt!]
  \subcaptionbox{$(\b,s)=(0.33,0.0133)$\label{fig1:a}}{\includegraphics[width=.33\linewidth]{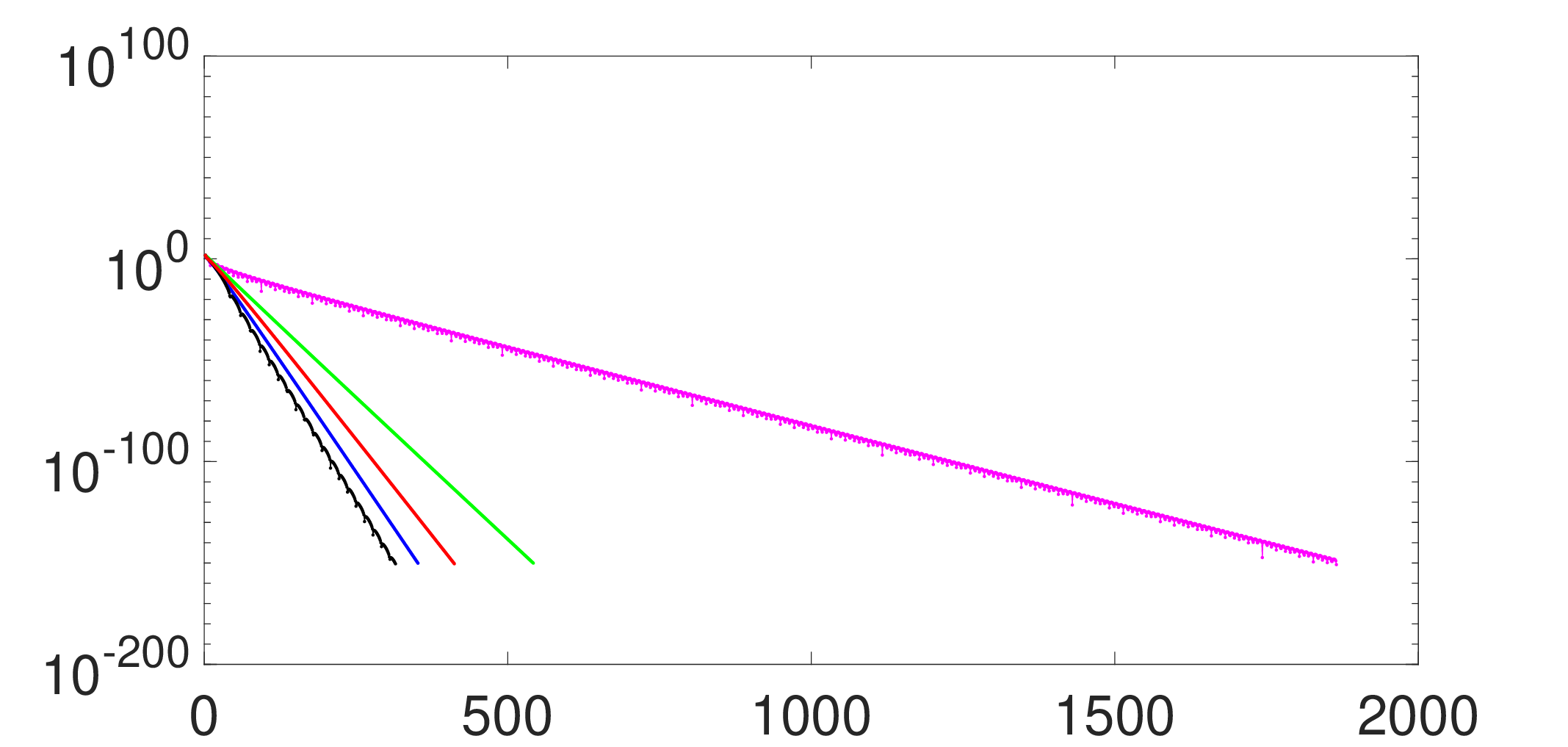}}\hfill%
   \subcaptionbox{$(\b,s)=(0.5,0.009)$\label{fig1:c}}{\includegraphics[width=.33\linewidth]{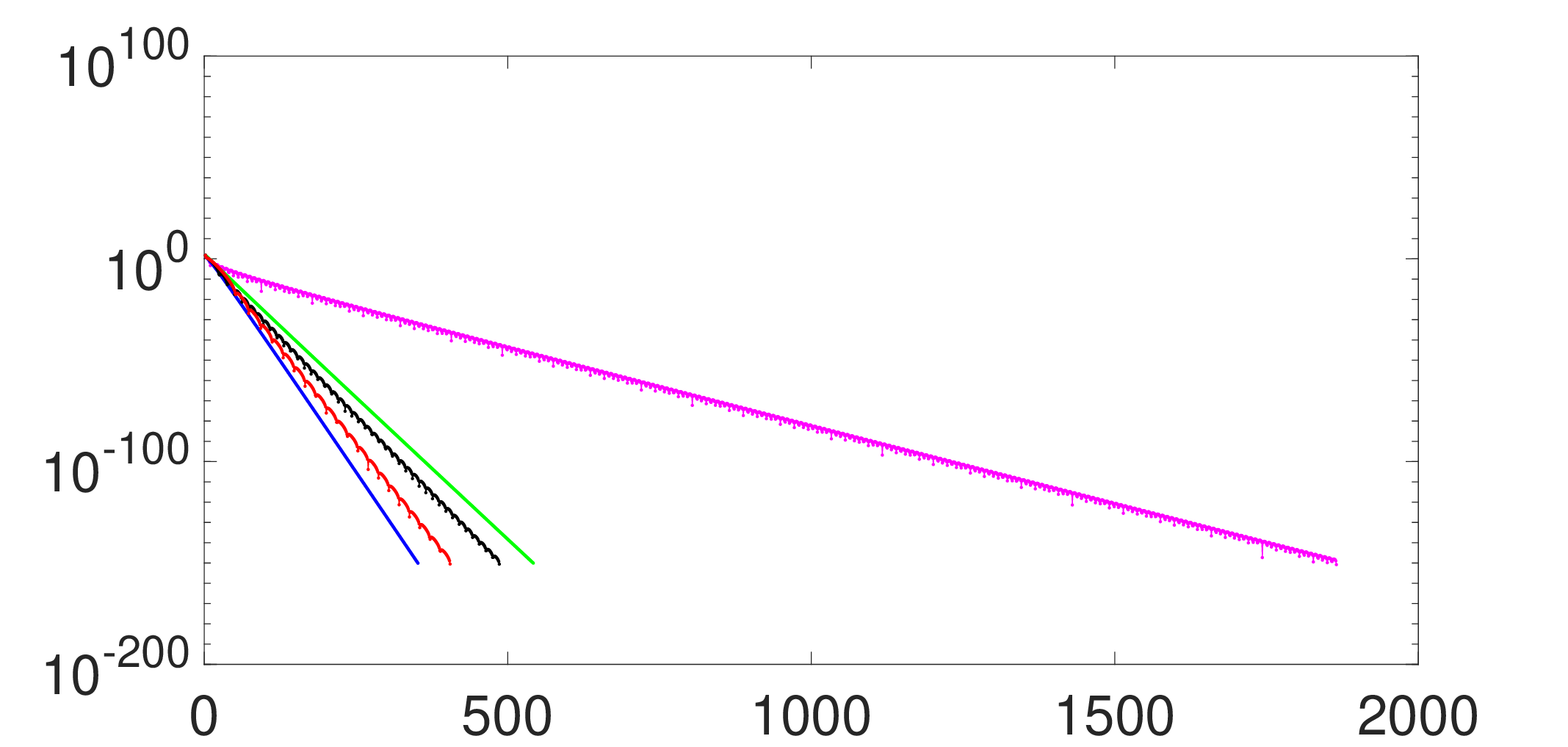}}\hfill%
  \subcaptionbox{$(\b,s)=(0.66,0.0067)$\label{fig1:e}}{\includegraphics[width=.33\linewidth]{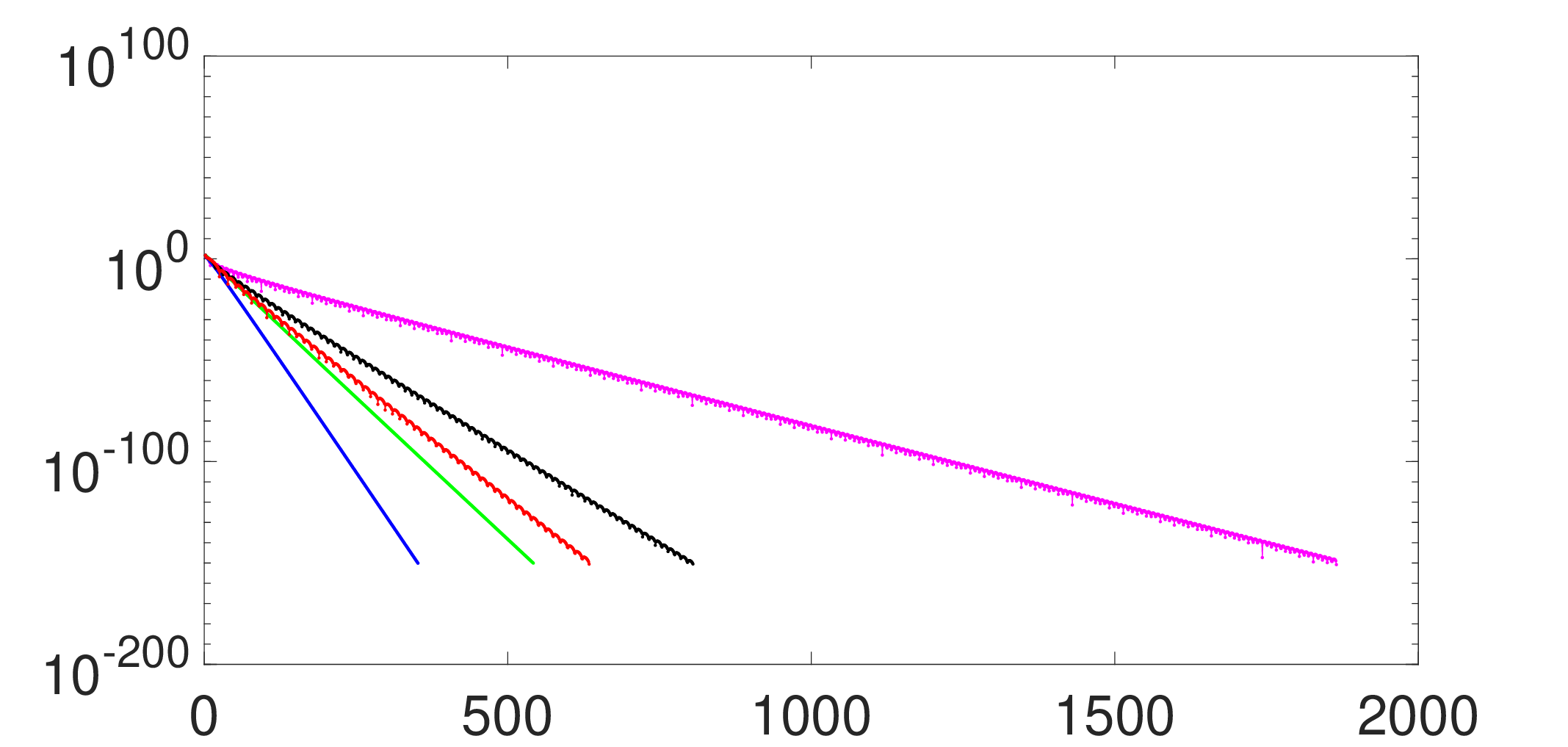}}\hfill%
\caption{Comparing the energy error $|g(x_n)-\min g|$  for  Algorithm \eqref{nest1} (magenta), Algorithm \eqref{neststr1} (blue), Algorithm \eqref{graddesc} (green), Algorithm \eqref{polyak1} (black) and Algorithm \eqref{generaldiscrete} (red), in the framework of the  minimization of the strongly convex function $g(x,y)=8x^2+50y^2$, by considering different stepsizes and different inertial coefficients}
\label{fig1}
\end{figure}

Further, we are also interested in the behaviour of the generated sequences $x_n$. To this end, we run the algorithms until the absolute error $\|x_n-x^*\|$ attains the value $10^{-150}$, where $x^*=(0,0)$ is the unique minimizer of the strongly convex function $g.$ The results are shown in Figure 2, where the horizontal axis measures the number of iterations and the vertical axis  shows the error $\|x_{n}-x^*\|$.
\begin{figure}[hbt!]
  \subcaptionbox{$(\b,s)=(0.33,0.0133)$\label{fig1:b}}{\includegraphics[width=.33\linewidth]{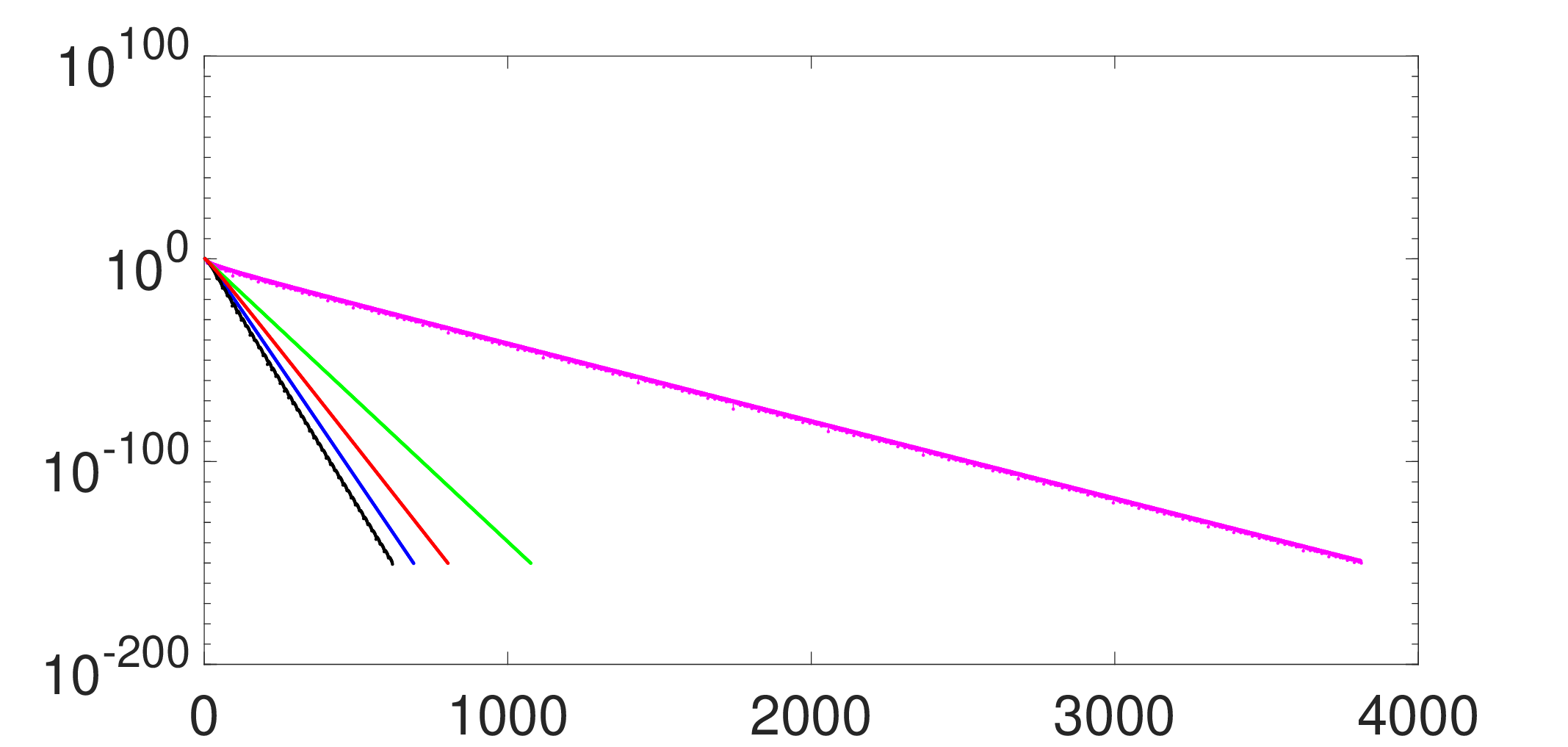}}\hfill%
 \subcaptionbox{$(\b,s)=(0.5,0.009)$\label{fig1:d}}{\includegraphics[width=.33\linewidth]{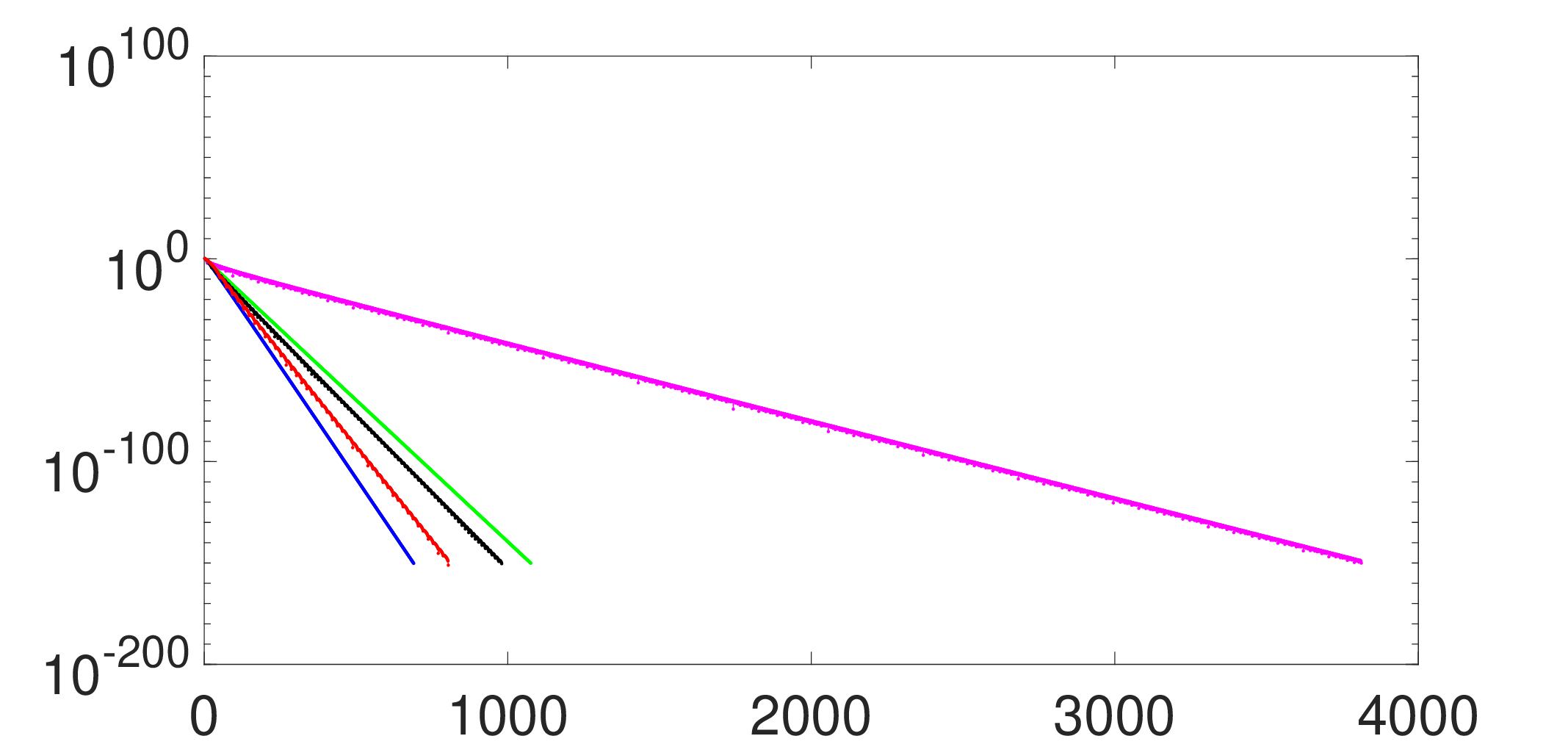}}\hfill%
  \subcaptionbox{$(\b,s)=(0.66,0.0067)$\label{fig1:f}}{\includegraphics[width=.33\linewidth]{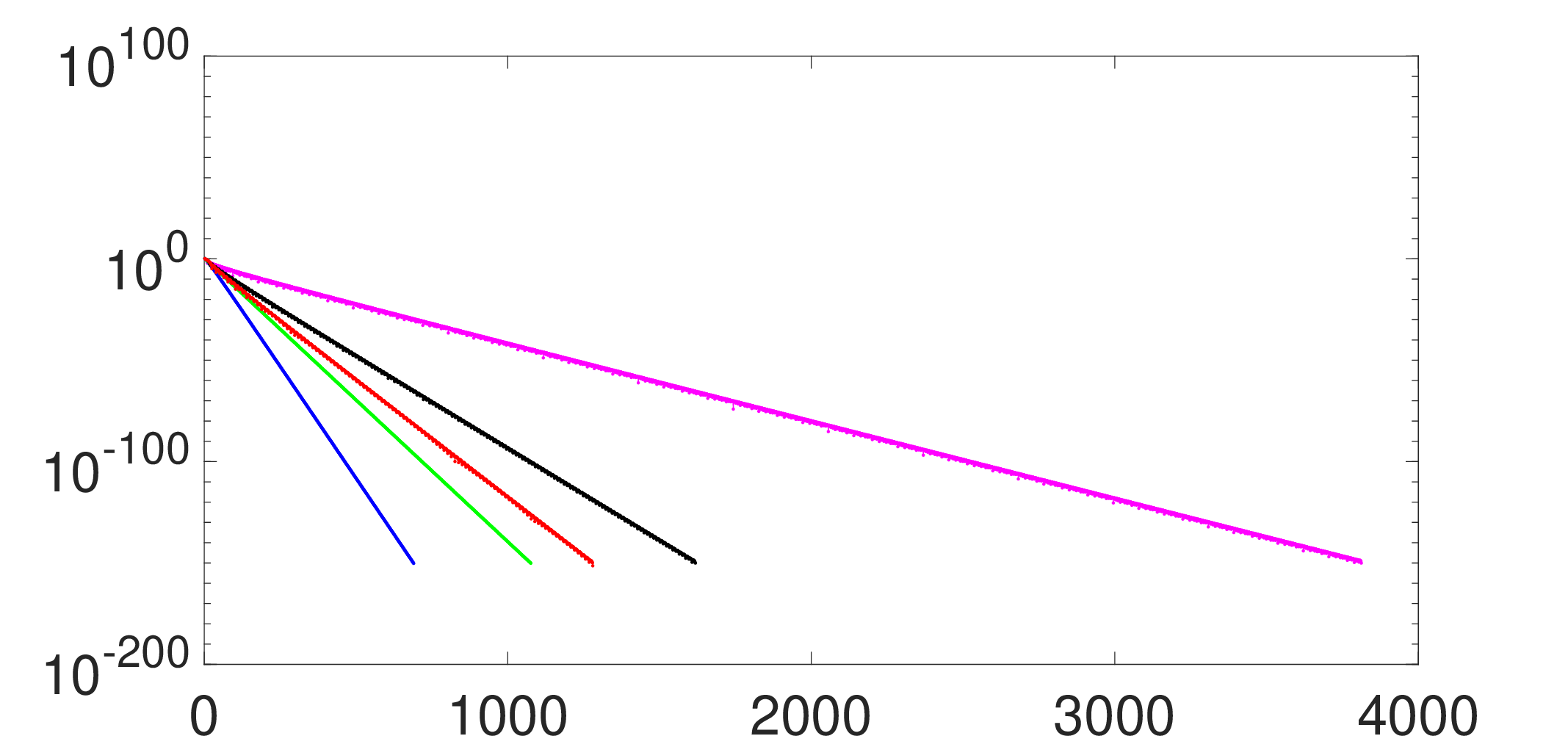}}\hfill%
\caption{Comparing the iteration error $\|x_n-x^*\|$  for  Algorithm \eqref{nest1} (magenta), Algorithm \eqref{neststr1} (blue), Algorithm \eqref{graddesc} (green), Algorithm \eqref{polyak1} (black) and Algorithm \eqref{generaldiscrete} (red), in the framework of the  minimization of the strongly convex function $g(x,y)=8x^2+50y^2$, by considering different stepsizes and different inertial coefficients}
\label{fig2}
\end{figure}

The experiments, depicted in Figure 1 and Figure 2, show that Algorithm \eqref{generaldiscrete}  has  a good behaviour since may outperform Algorithm \eqref{nest1},  Algorithm \eqref{graddesc}  and also Algorithm \eqref{polyak1}. However Algorithm \eqref{neststr1}  seems to have in all these cases a better behaviour.

\begin{remark} Nevertheless, for an appropriate choice of the parameters $\a$ and $\b$  Algorithm \eqref{generaldiscrete} outperforms Algorithm  \eqref{neststr1}. In order to sustain our claim observe that the inertial parameter in Algorithm \eqref{generaldiscrete} is $\frac{\b n}{n+\a}\approx \b$ when $n$ is big enough, (or $\a$ is very small). Consequently, if one takes $\b\approx \frac{\sqrt{L_g}-\sqrt{\mu}}{\sqrt{L_g}+\sqrt{\mu}}$, then  for $n$ big enough the inertial parameters in Algorithm \eqref{generaldiscrete} and Algorithm  \eqref{neststr1} are very close. However if $\b<\frac{1}{2}$ then the stepsize in Algorithm \eqref{generaldiscrete} is clearly better then the stepsize in Algorithm \eqref{neststr1}. This happens whenever $\frac{\sqrt{L_g}-\sqrt{\mu}}{\sqrt{L_g}+\sqrt{\mu}}<\frac12,$ that is $L_g< 9\mu.$
\end{remark}
 In the case of  the strongly convex function $g$ considered before, one has $\frac{\sqrt{L_g}-\sqrt{\mu}}{\sqrt{L_g}+\sqrt{\mu}}=\frac37\approx0.428.$ Therefore, in the following numerical experiment we consider Algorithm \eqref{polyak1}  and Algorithm \eqref{generaldiscrete} with $\b = 0.4$ and optimal admissible constant  stepsize $s=0.0119< \frac{2(1-0.6)}{L_g}=0.012$, and all the other instances  we let unchanged.  We run the simulations until the energy error $|g(x_{n})-\min g|$ and the absolute error $\|x_n-x^*\|$ attains the value $10^{-150},$ Figure 3 (a)-(b). Observe that in this case Algorithm \eqref{generaldiscrete} clearly outperforms Algorithm \eqref{neststr1}.
%\vskip0.5cm
\begin{figure}[hbt!]
\subcaptionbox{$(\b,s)=(0.4,0.0119)$\label{fig3:a}}{\includegraphics[width=.5\linewidth]{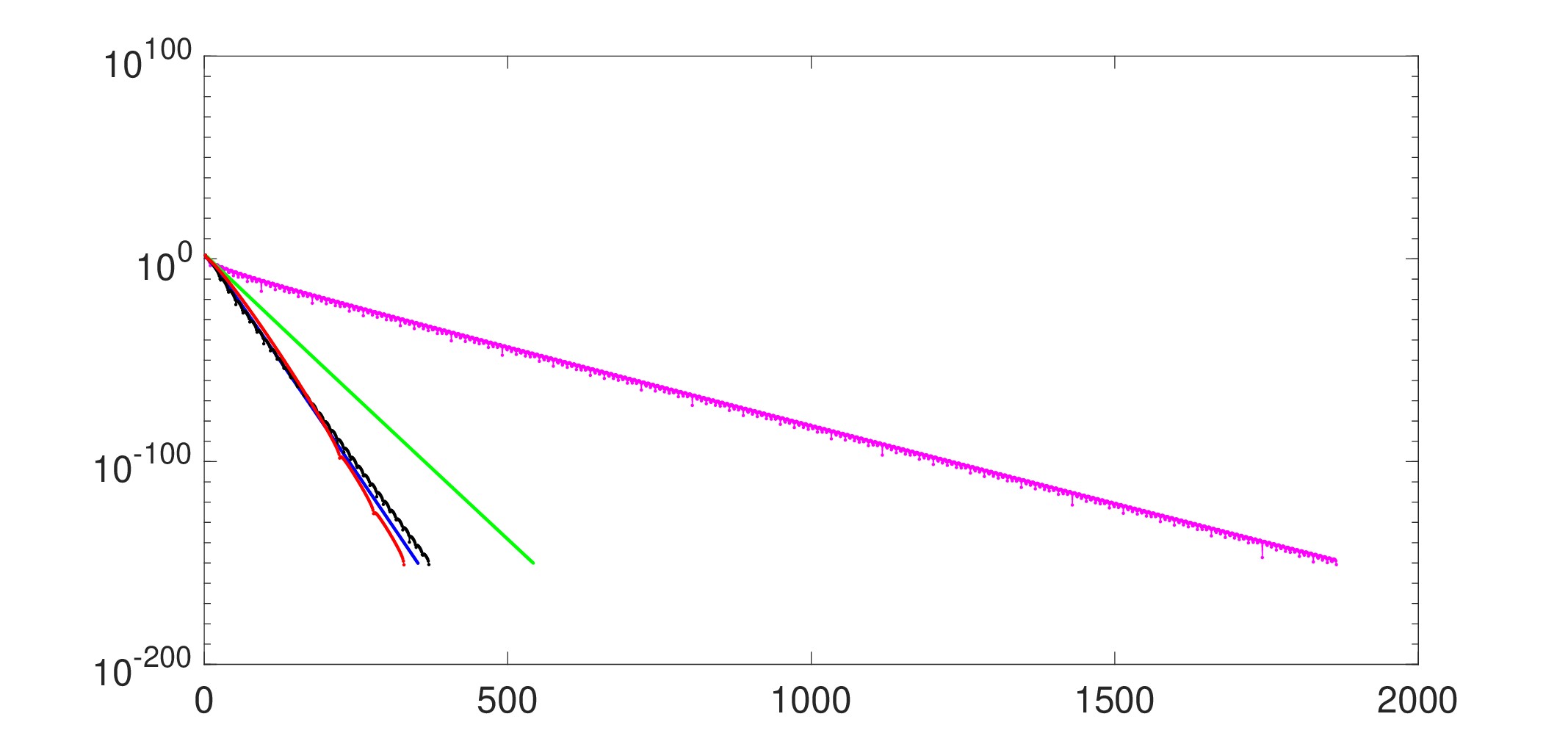}}\hfill%
  \subcaptionbox{$(\b,s)=(0.4,0.0119)$\label{fig3:b}}{\includegraphics[width=.5\linewidth]{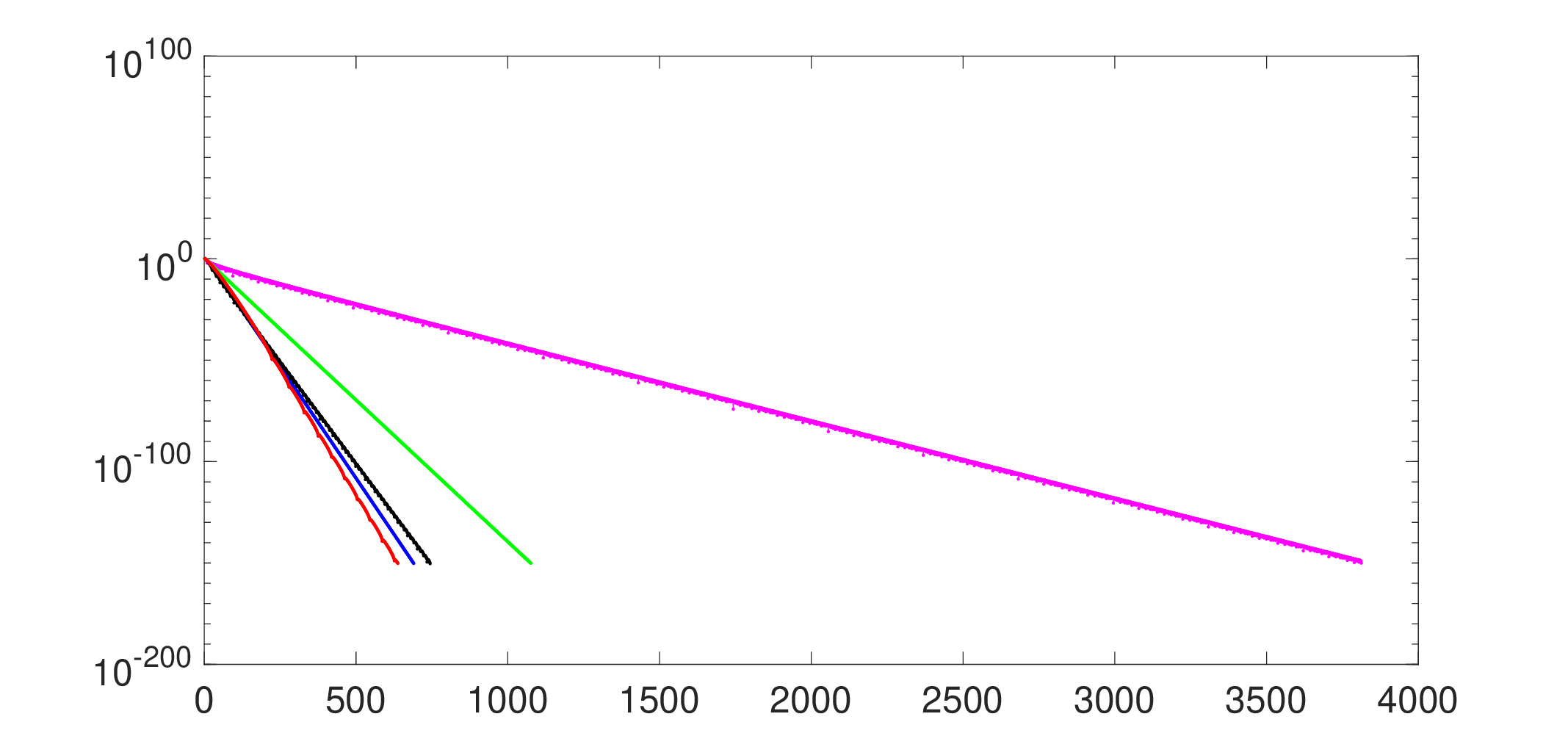}}%
\caption{ }%Comparing Nesterov's numerical method \eqref{nest1} and Polyak's numerical methods \eqref{polyak1} with Algorithm \eqref{generaldiscrete} in the framework of the  minimization of a strongly convex function, by considering different stepsizes and different inertial coefficients}
\label{fig3}
\end{figure}

{\bf 2.} In our second numerical experiment we consider a non-convex objective function
 $$g:\R^2\To\R,\,g(x,y)=x^2+y^2(1-x).$$
 Then, $\n g(x,y)=(2x-y^2,2y(1-x))$, which is obviously not Lipschitz continuous on $\R^2$, but the restriction of $\n g$ on $[-1,1]\times[-1,1]$  is Lipschitz continuous. Indeed, for $(x_1,y_1),(x_2,y_2)\in [-1,1]\times[-1,1]$ one has
 $$|(2x_1-y_1^2)-(2x_2-y_2^2)|\le2|x_1-x_2|+2|y_1-y_2|\le 2\sqrt{2}\|(x_1,y_1)-(x_2,y_2)\|$$
 and
 \begin{align}\nonumber |2y_1(1-x_1)-2y_2(1-x_2)|&=2|(y_1-y_2)-y_1(x_1-x_2)-x_2(y_1-y_2)|\le 2|x_1-x_2|+4|y_1-y_2|\\
 \nonumber&\le 4\sqrt{2}\|(x_1,y_1)-(x_2,y_2)\|,
 \end{align}
 consequently,
 \begin{align}\nonumber\|\n g(x_1,y_1)-\n g(x_2,y_2)\|&=\sqrt{|(2x_1-y_1^2)-(2x_2-y_2^2)|^2+|2y_1(1-x_1)-2y_2(1-x_2)|^2}\\
 \nonumber&\le\sqrt{ 8\|(x_1,y_1)-(x_2,y_2)\|^2+ 32\|(x_1,y_1)-(x_2,y_2)\|^2}\\
 \nonumber&=\sqrt{40}\|(x_1,y_1)-(x_2,y_2)\|.
 \end{align}

 Hence, one can consider that the Lipschitz constant of $g$ on $[-1,1]\times[-1,1]$ is $L_g=\sqrt{40}.$

 It is easy to see that the critical point of $g$ on $[-1,1]\times[-1,1]$ is $x^*=(0,0).$ Further, the Hessian of $g$ is
 $$\n^2 g(x,y)=\left(
                    \begin{array}{cc}
                      2 & -2y \\
                      -2y & 2-2x \\
                    \end{array}
                  \right)$$  which is indefinite on $[-1,1]\times[-1,1]$, hence $g$ is neither  convex nor concave on $[-1,1]\times[-1,1].$

Since   $\det\n^2 g(x^*)=4>0$ and $g_{xx}(x^*)=2>0$, we obtain that $x^*$ is a local minimum of $g$ and actually is the unique minimum of $g$ on $[-1,1]\times[-1,1]$. However, $x^*$ is not a global minimum of $g$ since for instance $g(2,3)=-9<0=g(x^*).$

In our following numerical experiments we will use different inertial parameters in order to compare Algorithm \eqref{generaldiscrete}, Algorithm \eqref{nest1} and Algorithm \eqref{polyak1}. In these experiments we run the algorithms until the energy error $|g(x_{n})-g(x^*)|$ attains the value $10^{-50}$  and the iterate error $\|x_n-x^*\|$ attains the value $10^{-50}$.

Since $L_g=\sqrt{40}$ we take in Algorithm \eqref{nest1} the stepsize $s=0.158\approx \frac{1}{\sqrt{40}}$.  In Algorithm \eqref{polyak1} and Algorithm \eqref{generaldiscrete} we consider the instances $$(\b,s)\in\{(0.33,0.210), (0.5,0.157), (0.66,0.107)\}.$$ Obviously for these values we have $\b\in(0,1)$ and $0<s<\frac{2(1-\b)}{L_g}.$ Further, we consider the  same starting points $x_0=x_{-1}=(0.5,-0.5)$ from $[-1,1]\times[-1,1]$.

The results are shown in Figure 4 (a)-(c) and Figure 5 (a)-(c), where the horizontal axes measure the number of iterations and the vertical axes  show the error $|g(x_{n})- g(x^*)|$  and the  error $\|x_n-x^*\|$, respectively.

\begin{figure}[hbt!]
  \subcaptionbox{$(\b,s)=(0.33,0.210)$\label{fig4:a}}{\includegraphics[width=.33\linewidth]{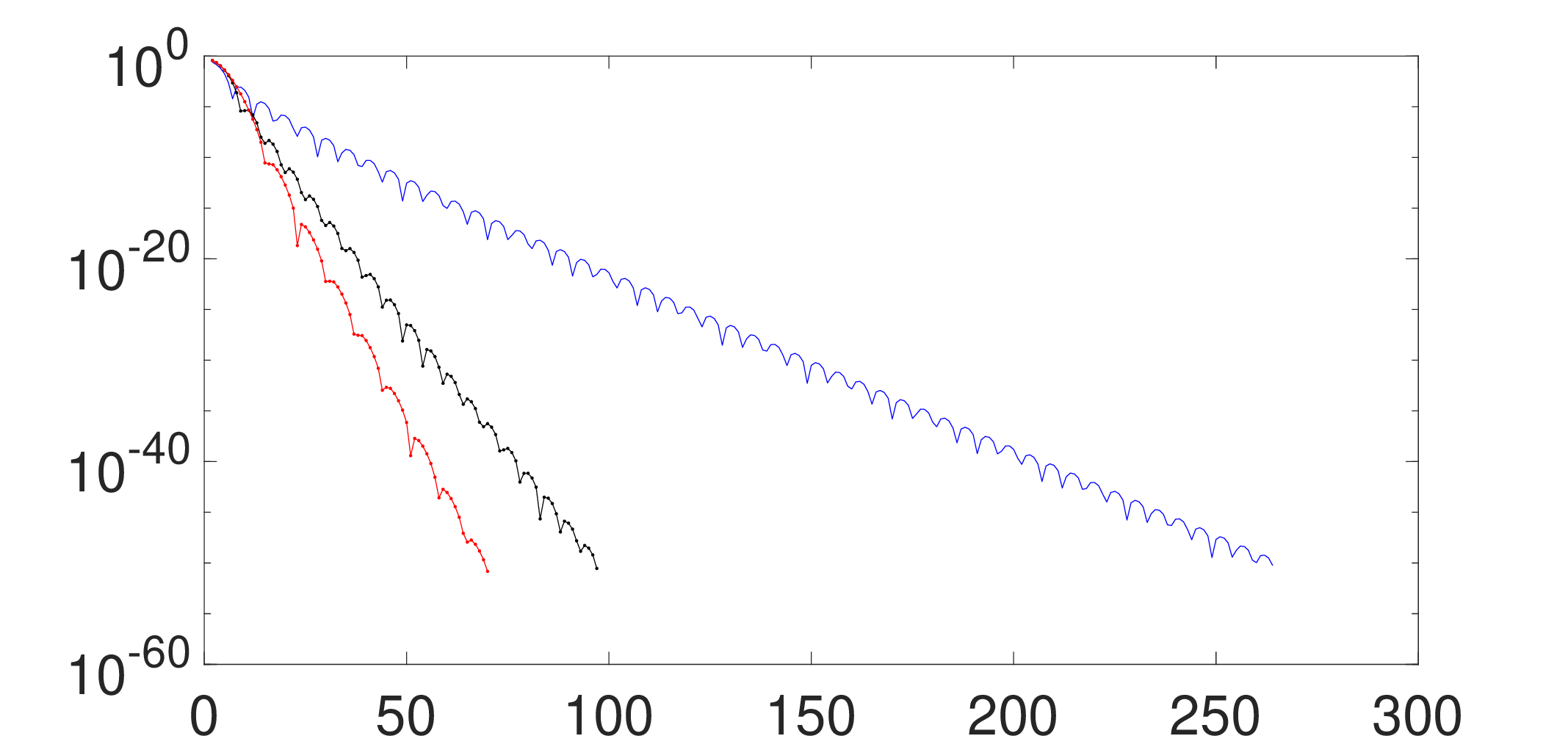}}\hfill%
  \subcaptionbox{$(\b,s)=(0.5,0.157)$\label{fig4:b}}{\includegraphics[width=.33\linewidth]{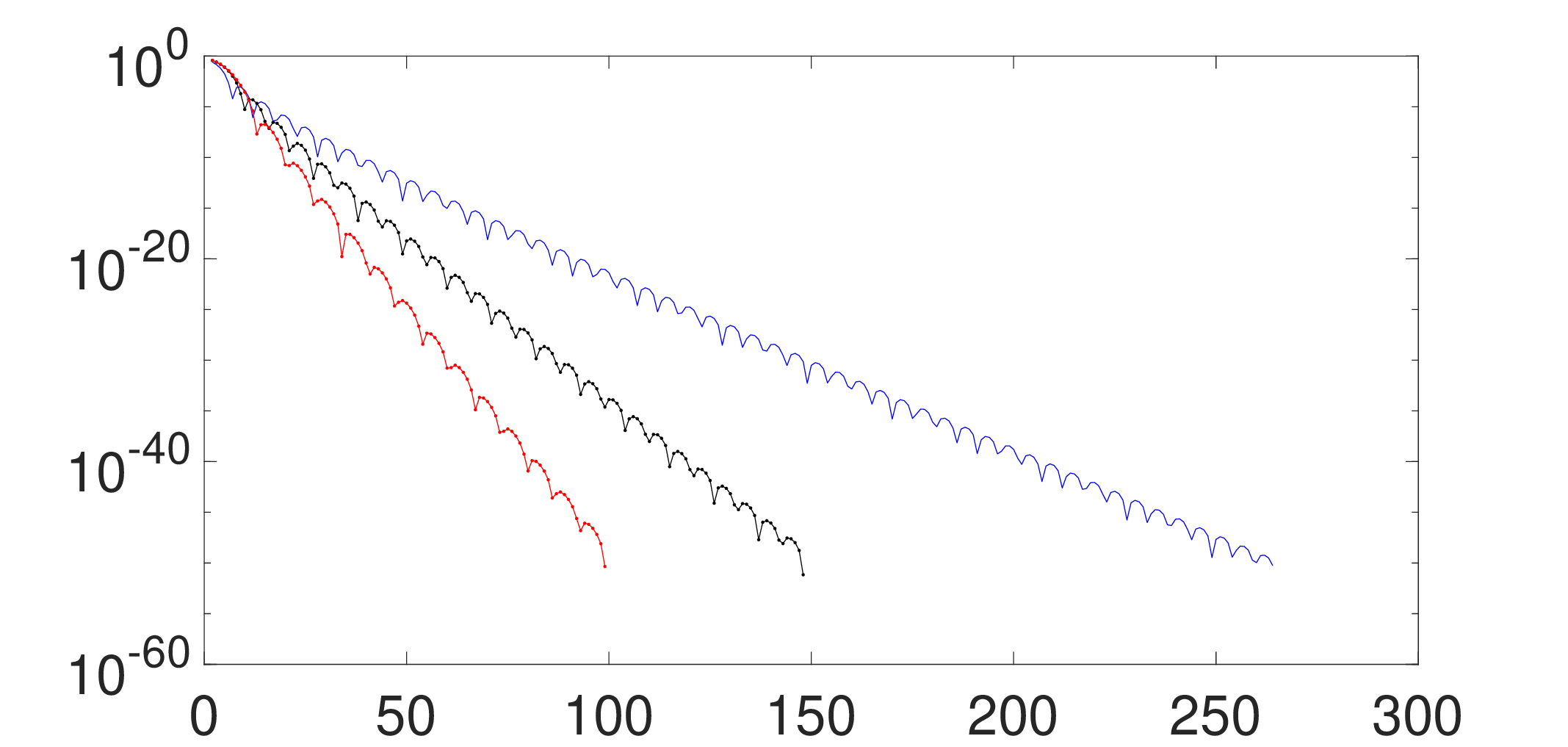}}\hfill%
  \subcaptionbox{$(\b,s)=(0.66,0.107)$\label{fig4:c}}{\includegraphics[width=.33\linewidth]{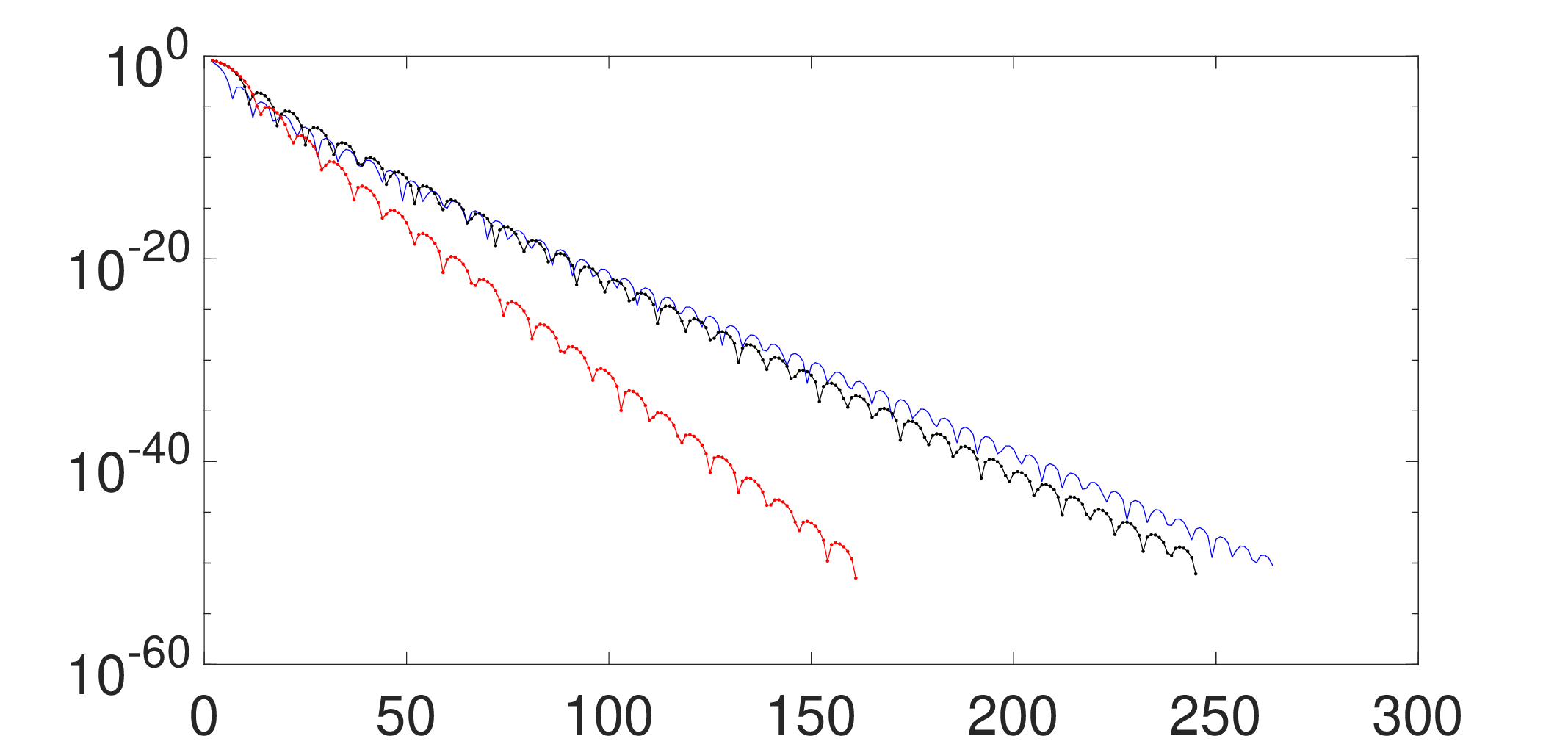}}\hfill%
  %\subcaptionbox{$x_0=x_{-1}=(3,0.4)$\label{fig4:d}}{\includegraphics[width=.25\linewidth]{beal4cor.eps}}\hfill%
  \caption{Comparing the energy error $|g(x_n)- g(x^*)|$  for  Algorithm \eqref{nest1} (blue),  Algorithm \eqref{polyak1} (black) and Algorithm \eqref{generaldiscrete} (red), in the framework of the  minimization of a non-convex function}\label{fig4}
\end{figure}

\begin{figure}[hbt!]
\subcaptionbox{$(\b,s)=(0.33,0.210)$\label{fig5:a}}{\includegraphics[width=.33\linewidth]{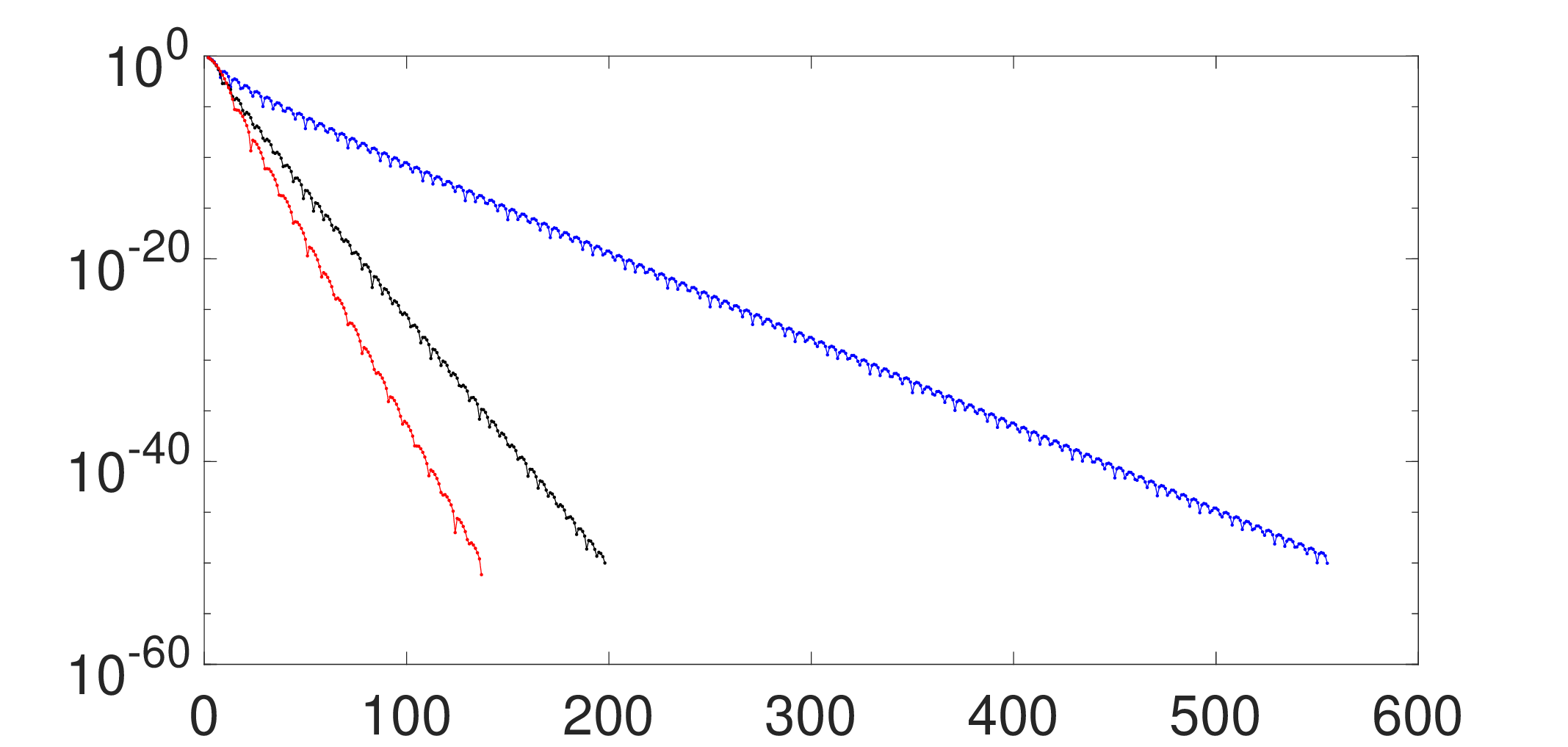}}\hfill%
  \subcaptionbox{$(\b,s)=(0.5,0.157)$\label{fig5:b}}{\includegraphics[width=.33\linewidth]{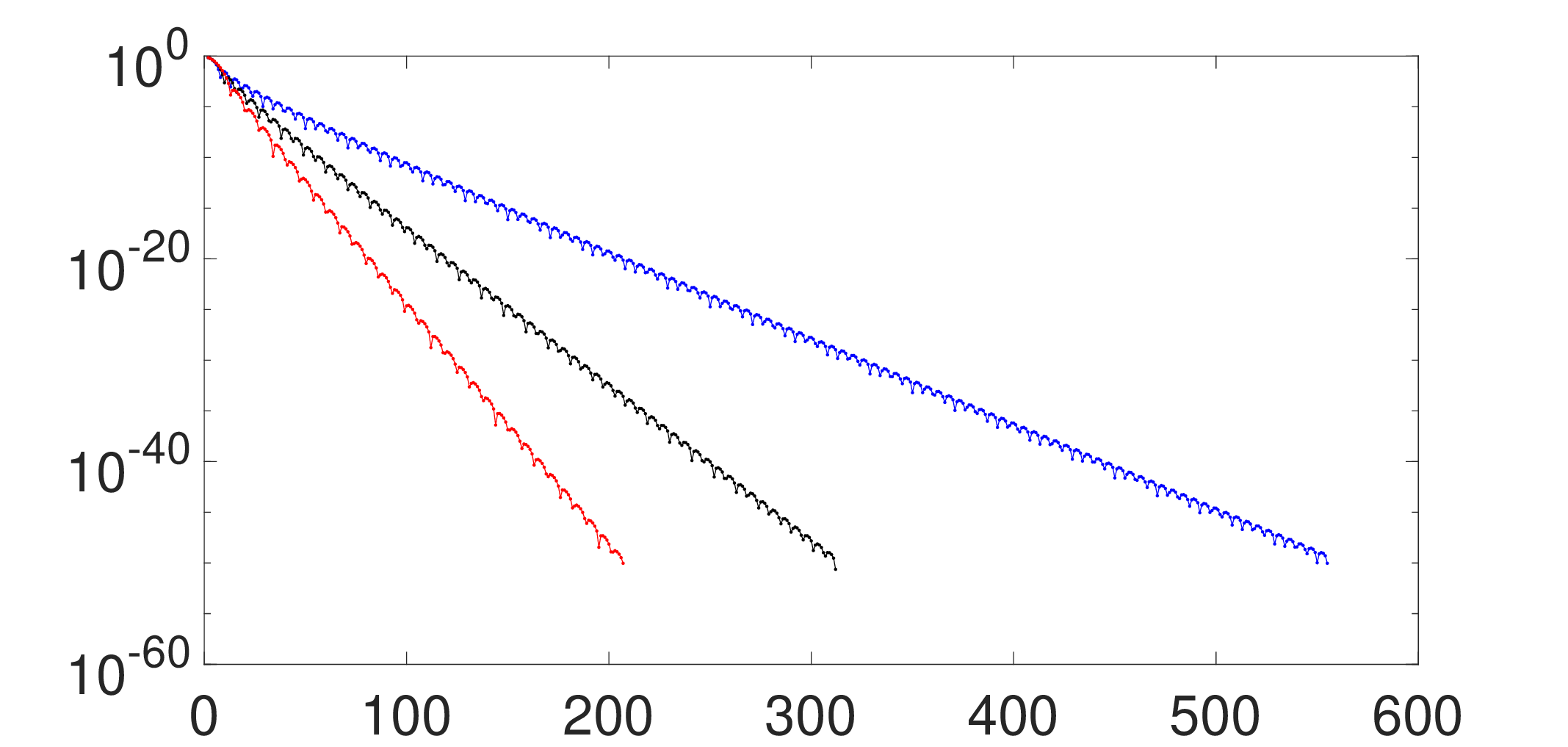}}\hfill%
  \subcaptionbox{$(\b,s)=(0.66,0.107)$\label{fig5:c}}{\includegraphics[width=.33\linewidth]{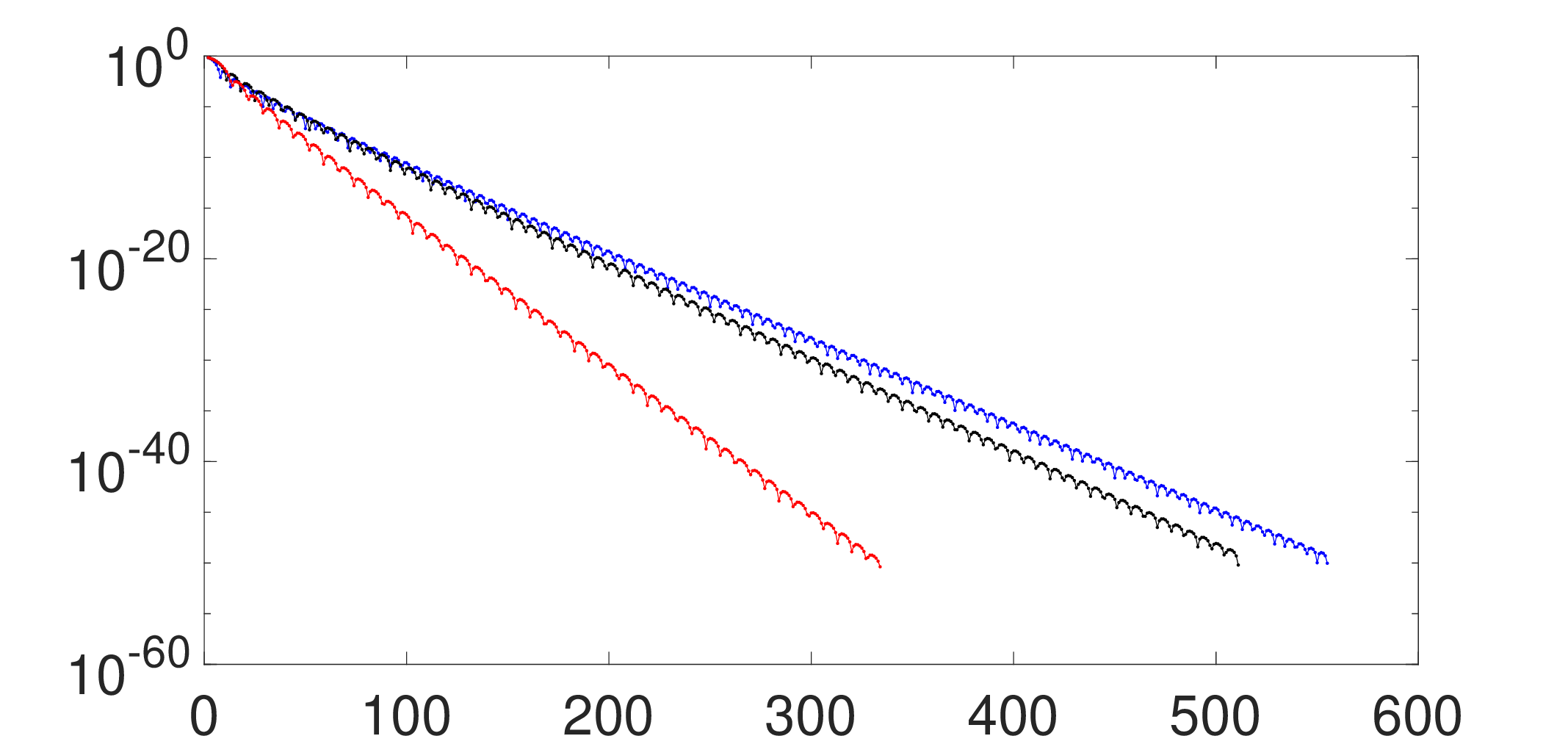}}\hfill%
    %\subcaptionbox{$x_0=x_{-1}=(1.1,0.8)$\label{fig5:d}}{\includegraphics[width=.25\linewidth]{bealiter4cor.eps}}\hfill%
\caption{Comparing the iteration error $\|x_n-x^*\|$ for  Algorithm \eqref{nest1} (blue),  Algorithm \eqref{polyak1} (black) and Algorithm \eqref{generaldiscrete} (red), in the framework of the  minimization of a non-convex function}
\label{fig5}
\end{figure}

Consequently, also for this non-convex function, Algorithm \eqref{generaldiscrete} outperforms both Algorithm \eqref{polyak1} and Algorithm \eqref{nest1}.

\section{Conclusions}

In this paper we show the convergence of a Nesterov type algorithm in a full non-convex setting by assuming that a regularization of the objective function satisfies the Kurdyka-{\L}ojasiewicz property. For this purpose we prove some abstract convergence results and we show that the sequences generated by our algorithm satisfy the conditions assumed in these abstract convergence results. More precisely, as a starting point we show a sufficient decrease property for the iterates generated by our algorithm and then via the KL property of a regularization of the objective in a cluster point of the generated sequence, we obtain the convergence of this sequence to this cluster point. Though our algorithm is asymptotically equivalent to Nesterov's accelerated gradient method, we cannot obtain full equivalence due to the fact that in order to obtain the above mentioned decrease property we cannot allow the inertial parameter, more precisely the parameter $\b$, to attain the value 1.
Nevertheless, we obtain finite, linear and sublinear  convergence rates   for the sequences generated by our numerical scheme but also for the  function values in these sequences, provided the objective function, or a regularization of the objective function, satisfies the {\L}ojasiewicz property with {\L}ojasiewicz exponent $\t\in\left[0,1\right).$

A related future research is the study of a  modified FISTA algorithm in a non-convex setting. Indeed, let $f:\R^m\To\ol\R$ be a  proper convex and lower semicontinuous function and let $g:\R^m\To\R$ be a (possible non-convex) smooth function with $L_g$ Lipschitz continuous gradient.
Consider the optimization problem
$$\inf_{x\in\R^m}f(x)+g(x).$$
We associate to this optimization problem the following proximal-gradient algorithm.
For $x_0,x_{-1}\in\R^m$ consider
\begin{equation}\label{fist}\left\{\begin{array}{lll}
\ds y_n=x_n+\frac{\b n}{n+\a}(x_n-x_{n-1}),\\
\\
\ds x_{n+1}=\prox\nolimits_{s f}(y_n-s\n g(y_n)),
\end{array}\right.
\end{equation}
where $\a>0,\,\b\in(0,1)$  and $0<s<\frac{2(1-\b)}{L_g}.$

Here
\begin{equation*}  \prox\nolimits_{s  f} : {\R^m} \rightarrow {\R^m}, \quad \prox\nolimits_{s f}(x)=\argmin_{y\in {\R^m}}\left \{f(y)+\frac{1}{2s}\|y-x\|^2\right\},
\end{equation*}
denotes the proximal point operator of the convex function $s  f$.

Obviously, when $f\equiv 0$ then \eqref{fist} becomes the numerical scheme \eqref{generaldiscrete} studied in the present paper.

We emphasize that \eqref{fist} has a  similar formulation as the modified FISTA algorithm studied by Chambolle and Dossal in \cite{ch-do2015} and the convergence of the generated sequences to a critical point of the objective function $f+g$  would open the gate for the study of FISTA type algorithms in a non-convex setting.
\vskip0.5cm
{\bf Acknowledgment}

The author is thankful  to two anonymous referees for their valuable remarks and suggestions which led to the improvement of the quality of the paper.

\appendix
\section{Appendix}
\subsection{Second order continuous dynamical systems that are modelling Algorithm \eqref{generaldiscrete}}\label{ssecA1}
In what follows we emphasize the connections between Algorithm \eqref{generaldiscrete} and the continuous dynamical systems \eqref{ee11} and \eqref{eee11}.

Consider \eqref{eee11} with the initial conditions $x(t_0)=u_0,\,\dot{x}(t_0)=v_0,\,u_0,v_0\in\R^m$ and the governing second order differential equation
$$\ddot{x}(t)+\left(\g+\frac{\a}{t}\right)\dot{x}(t)+\n g(x(t))=0,\,\g>0,\,\a\in\R.$$
We will use the time discretization presented in \cite{att-c-p-r-math-pr2018}, that is, we take the fixed stepsize $h> 0,$ and  consider $\b=1-\g h>0$,
 $t_n = \frac{1}{\b} nh$ and $x_n = x(t_n).$
 Then the implicit/explicit discretization of \eqref{ee11} leads to
 \begin{equation}\label{discr1}
 \frac{1}{h^2}(x_{n+1}-2x_n+x_{n-1})+\left(\frac{\g}{h}+\frac{\a\b}{nh^2}\right)(x_n-x_{n-1})+\n g(y_n)=0,
 \end{equation}
 where $y_n$ is a linear combination of $x_n$ and $x_{n-1}$ and will be defined below.

 Now, \eqref{discr1} can be rewritten as
$$x_{n+1}= x_n+\left(\b-\frac{\a\b}{ n}\right)(x_n-x_{n-1})-h^2\n g(y_n),$$
which suggest to choose $y_n$ in the form
$$y_n=x_n+\left(\b-\frac{\a\b}{ n}\right)(x_n-x_{n-1}).$$
However, for practical purposes,  it is convenient to work with the re-indexation $n\rightarrowtail n+\a$ and we obtain the following equivalent formulation
$$y_n=x_n+\frac{\b n}{ n+\a}(x_n-x_{n-1}).$$
Hence, by taking $h^2=s$ we get
$$x_{n+1}= x_n+\frac{\b n}{ n+\a}(x_n-x_{n-1})-s\n g(y_n),$$
which is exactly Algorithm \eqref{generaldiscrete}.

\begin{remark} Obviously, already the form $\b=1-\g h>0$ shows that $\b\in(0,1).$ We could not obtain Algorithm \eqref{generaldiscrete} via some similar  discretization of the continuous dynamical system \eqref{ee11} as the discretization method presented above. Nevertheless, we can show that \eqref{ee11} is the exact limit of Algorithm \eqref{generaldiscrete} in the sense of Su, Boyd and Cand\`es \cite{su-boyd-candes}.
\end{remark}

In what  follows we show that by choosing  appropriate values of $\b$, both the continuous second order dynamical systems \eqref{ee11} and  the continuous dynamical system \eqref{eee11} are the exact limit of  the numerical scheme \eqref{generaldiscrete}.

 To this end we take in \eqref{generaldiscrete} small step sizes and follow the same approach as Su, Boyd and Cand\`es in \cite{su-boyd-candes}, (see also \cite{BCL-AA} for similar approaches). For this purpose we rewrite \eqref{generaldiscrete} in the form
\begin{equation}\label{e411}
\frac{x_{n+1}-x_n}{\sqrt{s}}=\frac{\b n}{n+\a}\cdot\frac{x_n-x_{n-1}}{\sqrt{s}}-\sqrt{s}\n g(y_n) \ \forall n \geq 1
\end{equation}
and introduce the {\it Ansatz} $x_n\approx x(n\sqrt{s})$ for some twice continuously differentiable function $x : [0,+\infty) \rightarrow \R^n$. We let $n=\frac{t}{\sqrt{s}}$ and get  $x(t)\approx x_n,\,x(t+\sqrt{s})\approx x_{n+1},\,x(t-\sqrt{s})\approx x_{n-1}.$
Then, as the step size $s$ goes to zero, from the Taylor expansion of $x$ we obtain
$$\frac{x_{n+1}-x_n}{\sqrt{s}}=\dot{x}(t)+\frac12\ddot{x}(t)\sqrt{s}+o(\sqrt{s})$$
and
$$\frac{x_n-x_{n-1}}{\sqrt{s}}=\dot{x}(t)-\frac12\ddot{x}(t)\sqrt{s}+o(\sqrt{s}).$$

Further, since
$$\sqrt{s}\|\n g(y_n)-\n g(x_n)\|\le\sqrt{s} L_g\|y_n-x_n\|=\sqrt{s} L_g\left|\frac{\b n}{n+\a}\right|\|x_n-x_{n-1}\|
=o(\sqrt{s}),$$
it follows $\sqrt{s} \n g(y_n)=\sqrt{s}\n g(x_n)+ o(\sqrt{s})$. Consequently, \eqref{e411} can be written as
$$\dot{x}(t)+\frac12\ddot{x}(t)\sqrt{s}+ o(\sqrt{s})= $$
$$\frac{\b t}{t+\a\sqrt{s}}\left(\dot{x}(t)-\frac12\ddot{x}(t)\sqrt{s}+ o(\sqrt{s})\right)-\sqrt{s}\n g(x(t))+ o(\sqrt{s})$$
or, equivalently
$$(t+\a\sqrt{s})\left(\dot{x}(t)+\frac12\ddot{x}(t)\sqrt{s}+o(\sqrt{s})\right)=$$
$$\b t\left(\dot{x}(t)-\frac12\ddot{x}(t)\sqrt{s}+o(\sqrt{s})\right)-\sqrt{s}(t+\a\sqrt{s})\n g(x(t))+o(\sqrt{s}).$$
Hence,
\begin{equation}\label{ecom}
\frac12\left(\a\sqrt{s}+(1+\b)t\right)\ddot{x}(t)\sqrt{s}+\left((1-\b)t+\a\sqrt{s}\right)\dot{x}(t)+\sqrt{s}(t+\a\sqrt{s})\n g(x(t))=o(\sqrt{s}).
\end{equation}
Now, if we take $\b=1-\g {s}<1$ in \eqref{ecom} for some $\frac{1}{{s}}>\g>0$, we obtain
$$\frac12\left(\a \sqrt{s}+(2-\g{s})t\right)\ddot{x}(t)\sqrt{s}+\left(\g{s}t+\a\sqrt{s}\right)\dot{x}(t)+\sqrt{s}(t+\a\sqrt{s})\n g(x(t))=o(\sqrt{s}).$$

After dividing by $\sqrt{s}$ and letting $s\rightarrow 0$, we obtain
$$t\ddot{x}(t)+\a\dot{x}(t)+t\n g(x(t))=0,$$
which after division by $t$ gives \eqref{ee11}, that is,
$$\ddot{x}(t)+\frac{\a}{t}\dot{x}(t)+\n g(x(t))=0.$$

 %As we have mentioned, the latter system was studied in \cite{su-boyd-candes} for $\a=3.$
Similarly, by taking   $\b=1-\g\sqrt{s}<1$ in \eqref{ecom}, for some $\frac{1}{\sqrt{s}}>\g>0$, we obtain
$$\frac12\left(\a \sqrt{s}+(2-\g\sqrt{s})t\right)\ddot{x}(t)\sqrt{s}+\left(\g\sqrt{s}t+\a\sqrt{s}\right)\dot{x}(t)+\sqrt{s}(t+\a\sqrt{s})\n g(x(t))=o(\sqrt{s}).$$
 After dividing by $\sqrt{s}$ and letting $s\rightarrow 0$, we get
$$t\ddot{x}(t)+(\g t+\a)\dot{x}(t)+t\n g(x(t))=0,$$
which after division by $t$ gives \eqref{eee11}, that is,
$$\ddot{x}(t)+\left(\g+\frac{\a}{t}\right)\dot{x}(t)+\n g(x(t))=0.$$

%\vskip0.5cm

\end{document}